\documentclass[11pt,a4paper]{article}

\usepackage{amsmath}
\usepackage{amssymb,latexsym}
\usepackage{amsthm}
\usepackage[shortlabels]{enumitem}
\usepackage{array}
\usepackage[hang,flushmargin,bottom]{footmisc}
\usepackage[margin=0.35in,format=hang,font=small,labelfont=bf]{caption}
\usepackage[normalem]{ulem}
\usepackage{needspace}
\usepackage{graphicx,tikz}
\usepackage[colorlinks=true,urlcolor=blue,linkcolor=blue,citecolor=blue]{hyperref}
\usepackage{palatino}
\usepackage{eulervm}
\usepackage{tablefootnote}

\newtheorem{thm}{Theorem}[section]
\newtheorem*{thm*}{Theorem}
\newtheorem{prop}[thm]{Proposition}
\newtheorem{cor}[thm]{Corollary}
\newtheorem{lemma}[thm]{Lemma}

\newtheorem{obs}[thm]{Observation}
\newtheorem{conj}[thm]{Conjecture}
\newtheorem*{conj*}{Conjecture}

\theoremstyle{definition}

\newtheorem{defn*}{Definition}

\newtheorem*{example*}{Example}

\newtheorem*{comment*}{Comment}

\newcommand{\BBB}{\mathcal{B}}
\newcommand{\CCC}{\mathcal{C}}

\newcommand{\DDD}{\mathcal{D}}

\newcommand{\GGG}{\mathcal{G}}

\newcommand{\NNN}{\mathcal{N}}

\newcommand{\PPP}{\mathcal{P}}

\newcommand{\Free}{\operatorname{Free}}

\newcommand{\Label}{\operatorname{label}}
\newcommand{\Colour}{\operatorname{colour}}
\newcommand{\cwd}{\operatorname{cwd}}

\newcommand{\tree}{\operatorname{tree}}

\usetikzlibrary{calc, backgrounds}
\usetikzlibrary{decorations.markings}
\newcommand{\mnnodearray}[2]{ 
\foreach \i in {1,...,#1}
	\foreach \j in {1,...,#2}
		\node at (\i,\j) {};
}
\newcommand{\downhoriz}[2]{ 
\pgfmathtruncatemacro\rightcol{#1 + 1}
\foreach \i in {1,...,#2}
	\foreach \j in {\i,...,#2}
		\draw (#1,\j) -- (\rightcol,\i);
}
\newcommand{\upnohoriz}[2]{ 
\pgfmathtruncatemacro\rightcol{#1 + 1}
\pgfmathtruncatemacro\rowminus{#2 - 1}
\foreach \i in {1,...,\rowminus}
	{
	\pgfmathtruncatemacro\ip{\i+1}
	\foreach \j in {\ip,...,#2}
		\draw (#1,\i) -- (\rightcol,\j);
	}
}
\newcommand{\downnohoriz}[2]{ 
\pgfmathtruncatemacro\rightcol{#1 + 1}
\pgfmathtruncatemacro\rowminus{#2 - 1}
\foreach \i in {1,...,\rowminus}
	{
	\pgfmathtruncatemacro\ip{\i+1}
	\foreach \j in {\ip,...,#2}
		\draw (#1,\j) -- (\rightcol,\i);
	}
}
\newcommand{\horiz}[2]{ 
\pgfmathtruncatemacro\rightcol{#1 + 1}
\foreach \i in {1,...,#2}
		\draw (#1,\i) -- (\rightcol,\i);
}

\tikzstyle{vertex}=[circle, draw, fill=black,
                        inner sep=0pt, minimum width=4pt]
\tikzstyle{every node}=[circle, draw, fill=black,
                        inner sep=0pt, minimum width=4pt]

\title{A framework for minimal hereditary classes of graphs of unbounded clique-width}

\author{R. Brignall \qquad D. Cocks \footnote{Supported by the Engineering and Physical Sciences Research Council [EP/V520147/1].}\\
\small School of Mathematics and Statistics\\
\small The Open University, UK}

\begin{document}
\maketitle

\begin{abstract}
We create a framework for hereditary graph classes $\mathcal{G}^\delta$ built on a two-dimensional grid of vertices and edge sets defined by a triple $\delta=\{\alpha,\beta,\gamma\}$ of objects that define edges between consecutive columns, edges between non-consecutive columns (called bonds), and edges within columns. This framework captures a large family of minimal hereditary classes of graphs of unbounded clique-width, some previously identified and many new ones, although we do not claim this includes all such classes.

We show that a graph class $\mathcal{G}^\delta$ has unbounded clique-width if and only if a certain parameter $\mathcal{N}^\delta$ is unbounded. 
We further show that $\mathcal{G}^\delta$ is minimal of unbounded clique-width (and, indeed, minimal of unbounded linear clique-width) if another parameter $\mathcal{M}^\beta$ is bounded, and also $\delta$ has defined recurrence characteristics. Both the parameters $\mathcal{N}^\delta$ and $\mathcal{M}^\beta$ are properties of a triple $\delta=(\alpha,\beta,\gamma)$, and measure the number of distinct neighbourhoods in certain auxiliary graphs.

Throughout our work, we introduce new methods to the study of clique-width, including the use of Ramsey theory in arguments related to unboundedness, and explicit (linear) clique-width expressions for subclasses of minimal classes of unbounded clique-width.

\end{abstract}

%
%
%
%
%
%
\section{Introduction}

Until $4$ years ago only a couple of examples of minimal hereditary classes of unbounded clique-width had been identified, see Lozin~\cite{lozin:minimal-classes:}. However, more recently many more such classes have been identified, in Atminas, Brignall, Lozin and Stacho~\cite{abls:minimal-classes-of:}, Collins, Foniok, Korpelainen, Lozin and Zamaraev~\cite{collins:infinitely-many:}, Dawar and Sankaran~\cite{dawar:clique_width:} and most recently the current authors demonstrated an uncountably infinite family of minimal hereditary classes of unbounded clique-width in \cite{brignall_cocks:uncountable:}. 

This paper brings together all but one of these examples into a single consistent framework. The framework consists of hereditary graph classes constructed by taking the finite induced subgraphs of an infinite graph $\PPP^\delta$ whose vertices form a two-dimensional array and whose edges are defined by three objects, collectively denoted as a triple $\delta=(\alpha,\beta,\gamma)$. Though we defer full definitions until Section~\ref{prelim}, the components of the triple define edges between consecutive columns ($\alpha$), between non-consecutive columns ($\beta$ `bonds'), and within columns ($\gamma$) as follows.

\begin{enumerate}[label=(\alph*)]
\item $\alpha$ is an infinite word from the alphabet $\{0,1,2,3\}$. The four types of $\alpha$-edge sets between consecutive columns can be described as a matching ($0$), the complement of a matching ($1$), a chain ($2$) and the complement of a chain ($3$), (illustrated in Figure \ref{Hgraph}). 
\item $\beta$ is a symmetric subset of pairs of natural numbers $(x,y)$. If $(x,y) \in \beta$ then every vertex in column $x$ is adjacent to every vertex in column $y$.
\item $\gamma$ is an infinite binary word. If the $j$-th letter of $\gamma$ is $0$ then vertices in column $j$ form an independent set and if it is $1$ they form a clique.   
\end{enumerate}

We show that these hereditary graph classes $\mathcal{G}^\delta$ have unbounded clique-width if and only if a parameter $\mathcal{N}^\delta$ measuring the number of distinct neighbourhoods between any two rows of the grid, is unbounded -- see Theorem \ref{thm-0123unbound}. We denote $\Delta$ as the set of $\delta$-triples for which $\mathcal{G}^\delta$ has unbounded clique-width.

Furthermore, we define a subset $\Delta_{min} \subset \Delta$ such that if $\delta \in \Delta_{min}$ the hereditary graph class $\mathcal{G}^\delta$ is minimal both of unbounded clique-width and of unbounded \emph{linear} clique-width (Definitions in Section \ref{sect_cwd} and result Theorem \ref{thm-0123minim}). 
Referring to $\delta^*=\delta_{[a,a+b]}$ as a \emph{factor of $\delta$} being a subset  of $\delta$ defining all edges between  vertices in columns $a, a+1,\dots, a+b$, these 'minimal' $\delta$-triples are characterised by:
\begin{enumerate}[label=(\alph*)]
\item $\delta \in \Delta$, 
\item $\delta$ is \emph{$\mathcal{N}^{\delta}$-bounded recurrent} (i.e. any factor $\delta^*$ of $\delta$ repeats an infinite number of times, and the subgraphs induced on the columns between two consecutive disjoint copies of $\delta^*$ (the $\delta$-factor `gap') have bounded $\mathcal{N}^{\delta}$ (always true for almost periodic $\delta$)), and
\item a bound on a parameter $\mathcal{M}^\beta$ defined by the bond set $\beta$, which is a measure of the number of distinct neighbourhoods between intervals of a single row.  
\end{enumerate}

All but one hereditary graph classes previously shown to be minimal of unbounded clique-width  fit this grid framework i.e.\ they are defined by a $\delta$-triple in $\Delta_{min}$. This is demonstrated in Table~\ref{table-classes} which shows their corresponding $\delta=(\alpha,\beta,\gamma)$ values from the framework. The only minimal class so far discovered not in the table is \emph{power graphs} \cite{dawar:clique_width:}, a class built on a single path rather than a two dimensional grid.

\begin{table}[ht!]
\renewcommand{\arraystretch}{1.5}
\centering
\begin{tabular}{|m{4.75cm}|m{3.75cm}|m{4.5cm}|m{1.0cm}|} 
\hline
\textbf{Name}  &\textbf{$\alpha$} &\textbf{$\beta$} \text{      } \textbf{($x,y \in \mathbb{N}$)} &\textbf{$\gamma$}\\
\hline
Bipartite permutation \cite{lozin:minimal-classes:} &$2^\infty$ &$\emptyset$ &$0^\infty$ \\
\hline
Unit interval \cite{lozin:minimal-classes:} &$2^\infty$ &$\emptyset$ &$1^\infty$\\
\hline
Bichain \cite{abls:minimal-classes-of:} &$(23)^\infty$ &$(2x,2x+2y+1)$ &$0^\infty$\\
\hline
Split permutation \cite{abls:minimal-classes-of:} &$(23)^\infty$ &$(2x,y): y>2x+1$ &$(01)^\infty$\\
\hline
$\alpha \in \{0,1\}$ \cite{collins:infinitely-many:} &periodic&$\emptyset$ &$0^\infty$\\
\hline
$\alpha \in \{0,1,2,3\}$ \cite{brignall_cocks:uncountable:} &recurrent \tablefootnote{A set of minimal classes $\Gamma$ defined by an infinite word $\alpha$ which is recurrent over the alphabet $\{0,1,2,3\}$ and for which the 'gap' factors have a  bounded number of non-zero letters (including all almost periodic $\alpha$)} &$\emptyset$ &$0^\infty$\\
\hline
\end{tabular}
\caption{Hereditary graph classes proven to be minimal of unbounded clique-width}
\label{table-classes}
\end{table}

The viewpoint provided by our framework offers a fuller understanding of the landscape of (the uncountably many known) minimal hereditary classes of unbounded clique-width. This landscape is in stark contrast to the situation for downwards-closed sets of graphs under different orderings and with respect to other parameters.
For example, planar graphs are the unique minimal minor-closed class of graphs of unbounded treewidth (see Robertson and Seymour~\cite{robertson:graph-minors-v:}), and circle graphs are the unique minimal vertex-minor-closed class of unbounded rank-width (or, equivalently, clique-width) -- see Geelen, Kwon, McCarty and Wollan~\cite{geelen:the-grid-theorem:}. 
Nevertheless, clique-width is more compatible with hereditary classes of graphs than treewidth: if $H$ is an induced subgraph of $G$, then the clique-width of $H$ is at most the clique-width of $G$, but the same does not hold in general for treewidth.

Our focus on the minimal classes of unbounded clique-width is due to the following fact: any graph property expressible in MSO$_1$ logic has a linear time algorithm on graphs with bounded clique-width, see Courcelle, Makowsky and Rotics~\cite{courcelle:linear-time-sol:}. 
As it happens, any proper subclass of a minimal class from our framework also has bounded \emph{linear} clique-width. 
However, beyond our framework there do exist classes that have bounded clique-width but unbounded linear clique-width, see~\cite{alecu:lin-clique-width:} and~\cite{bkv:lcw-cographs:}.

After introducing the necessary definitions in Section~\ref{prelim}, the rest of this paper is organised as follows.

We set out in Section~\ref{Sect:Unbounded} our proof determining which hereditary classes $\mathcal{G}^\delta$ have unbounded clique-width. Proving a class has unbounded clique-width is done from first principles, using a new method, by identifying a lower bound for the number of labels required for a clique-width expression for an $n \times n$ square graph, using distinguished coloured vertex sets and showing such sets always exist for big enough $n$ using Ramsey theory. For those classes which have bounded clique-width, we prove this by providing a general clique-width expression for any graph in the class, using a bounded number of labels. 

In Section~\ref{Sect:Minimal} we prove that the class $\mathcal{G}^\delta$ is minimal of unbounded clique-width if $\delta \in \Delta_{min}$. To do this we introduce an entirely new method of 'veins and slices', partitioning the vertices of an arbitrary graph in a proper subclass of $\GGG^\delta$  into sections we call 'panels' using vertex colouring. We then create a recursive linear clique-width expression to construct these panels in sequence, allowing recycling of labels each time a new panel is constructed, so that an arbitrary graph can be constructed with a bounded number of labels. 

Previous papers on minimal hereditary graph classes of unbounded clique-width have focused mainly on bipartite graphs. 
The introduction of $\beta$-bonds and $\gamma$-cliques has significantly broadened the scope of proven minimal classes. 
 
In Section~\ref{Sect:examples} we provide some examples of new hereditary graph classes that are minimal of unbounded clique-width revealed by this approach.  Finally, in Section~\ref{Sect-Conclude}, we discuss where the investigation of minimal classes of unbounded clique-width might go next. 

%
%
%
%
%
\section{Preliminaries} \label{prelim}

\subsection{Graphs - General}

A graph $G=(V,E)$ is a pair of sets, vertices $V=V(G)$ and edges $E=E(G)\subseteq V(G)\times V(G)$. Unless otherwise stated, all graphs in this paper are simple, i.e. undirected, without loops or multiple edges. 

If vertex $u$ is adjacent to vertex $v$ we write $u \sim v$ and if $u$ is not adjacent to $v$ we write $u \not\sim v$. We denote $N(v)$ as the neighbourhood of a vertex $v$, that is, the set of vertices adjacent to $v$. A set of vertices is \emph{independent} if no two of its elements are adjacent and is a \emph{clique} if all the vertices are pairwise adjacent. We denote a clique with $r$ vertices as $K^r$ and an independent set of $r$ vertices as $\overline{K^r}$. A graph is \emph{bipartite} if its vertices can be partitioned into two independent sets, $V_1$ and $V_2$, and is \emph{complete bipartite} if, in addition, each vertex of $V_1$ is adjacent to each vertex of $V_2$.  

We will use the notation $H \le G$ to denote graph $H$ is an \emph{induced subgraph} of graph $G$, meaning $V(H) \subseteq V(G)$ and two vertices of $V(H)$ are adjacent in $H$ if and only if they are adjacent in $G$. We will denote the subgraph of $G=(V,E)$ induced by the set of vertices $U \subseteq V$ by $G[U]$. If graph $G$ does not contain an induced subgraph isomorphic to $H$ we say that $G$ is \emph{$H$-free}.

A class of graphs $\CCC$ is \emph{hereditary} if it is closed under taking induced subgraphs, that is $G \in \CCC$ implies $H \in \CCC$ for every induced subgraph $H$ of $G$. It is well known that for any hereditary class $\CCC$ there exists a unique (but not necessarily finite) set of minimal forbidden graphs $\{H_1, H_2, \dots\}$ such that $\CCC= \Free(H_1, H_2, \dots)$ (i.e. any graph $G \in \CCC$ is $H_i$-free for $i=1,2,\dots$). We will use the notation $\CCC \subseteq \GGG$ to denote that $\CCC$ is a  \emph{hereditary subclass} of hereditary graph class $\GGG$ ($\CCC \subsetneq \GGG$ for a proper subclass).

An \emph{embedding} of graph $H$ in graph $G$ is an injective map $\phi:V(H) \rightarrow V(G)$ such that the subgraph of $G$ induced by the vertices $\phi(V(H))$ is isomorphic to $H$. In other words, $vw \in E(H)$ if and only if $\phi(v) \phi(w) \in E(G)$. If $H$ is an induced subgraph of $G$ then this can be witnessed by one or more embeddings.

Given a graph $G=(V,E)$ and a subset of vertices $U \subseteq V$, two vertices of $U$ will be called \emph{$V\setminus U$-similar} if they have the same neighbourhood in $V\setminus U$. Thus $V\setminus U$-similarity is an equivalence relation. The number of such equivalence classes of $U$ in $G$ will be denoted $\mu(G,U)$.   A special case is when all the equivalence classes are singletons when we call $U$ a \emph{distinguished vertex set}.

A \emph{distinguished pairing} $\{U,W\}$ of size $r$ of a graph $G=(V,E)$ is a pair of vertex subsets $U=\{u_i\} \subseteq V$  and $W=\{w_i\} \subseteq V\setminus U$ with $|U|=|W|=r$ such that the vertices in $U$ have pairwise different neighbourhoods in $W$ (but not necessarily vice-versa). A distinguished pairing is \emph{matched} if the vertices of $U$ and $W$ can be paired $(u_i, w_i)$ so that  $u_i \sim w_i$ for each $i$, and is \emph{unmatched} if the vertices of $U$ and $W$ can be paired $(u_i, w_i)$ so that  $u_i \not\sim w_i$ for each $i$. Clearly the set $U$ of a distinguished pairing $\{U,W\}$ is a distinguished vertex set of $G[U \cup W]$ which gives us the following: 

\begin{prop}\label{prop-pair}
If $\{U,W\}$ is a distinguished pairing of size $r$ in graph $G$ then $\mu(G[U \cup W],U)=r$.
\end{prop}

\subsection{\texorpdfstring{$\GGG^\delta$}{\GGG{}} hereditary graph classes}

The graph classes we consider are all formed by taking the set of finite induced subgraphs of an infinite graph defined on a grid of vertices. We start by defining an infinite empty graph $\PPP$ with vertices 
\[V(\PPP) = \{v_{i,j} : i, j \in \mathbb{N}\}.\] 
We use Cartesian coordinates throughout this paper. Hence, we think of $\PPP$ as an infinite two-dimensional array in which $v_{i,j}$ represents the vertex in the $i$-th column (counting from the left) and the $j$-th row (counting from the bottom). Hence vertex $v_{1,1}$ is in the bottom left corner of the grid and the grid extends infinitely upwards and to the right. The $i$-th column of $\PPP$ is the set $C_i = \{ v_{i,j} : j \in \mathbb{N}\}$, and the $j$-th row of $\PPP$ is the set $R_j = \{ v_{i,j} : i \in \mathbb{N}\}$. Likewise, the collection of vertices in columns $i$ to $j$ is denoted $C_{[i,j]}$. 

We will add edges to $\PPP$ using a triple $\delta$ of objects that define the edges between consecutive columns, edges between non-consecutive columns and edges within each column. 

We refer to a (finite or infinite) sequence of letters chosen from a finite alphabet as a \emph{word}.  We denote by $\omega_i$ the $i$-th letter of the word $\omega$. A \emph{factor} of $\omega$ is a contiguous subword $\omega_{[i,j]}$ being the sequence of letters from the $i$-th to the $j$-th letter of $\omega$. If $a$ is a letter from the alphabet we will denote $a^\infty$ as the infinite word $aaa\dots$, and if $a_1\dots a_n$ is a finite sequence of letters from the alphabet then we will denote $(a_1\dots a_n)^\infty$ as the infinite word consisting of the infinite repetition of this factor. 

The \emph{length} of a word (or factor) is the number of letters the word contains.

An infinite word $\omega$ is \emph{recurrent} if each of its factors occurs in it infinitely many times. We say that $\omega$ is \emph{almost periodic} (sometimes called \emph{uniformly recurrent} or \emph{minimal}) if for each factor $\omega_{[i,j]}$ of $\omega$ there exists a constant $\mathcal{L}(\omega_{[i,j]})$ such that every factor of $\omega$ of length at least $\mathcal{L}(\omega_{[i,j]})$ contains $\omega_{[i,j]}$ as a factor. Finally, $\omega$ is \emph{periodic} if there is a positive integer $p$ such that $\omega_k=\omega_{k+p}$ for all $k$. Clearly, every periodic word is almost periodic, and every almost periodic word is recurrent.

A \emph{bond-set} $\beta$ is a symmetric subset of $\{(x,y) \in \mathbb{N}^2, |x-y| > 1\}$. For a set $Q \subseteq \mathbb{N}$ we write $\beta_Q$ to mean the subset of $\beta$-bonds $\{(x,y) \in \beta : x,y \in Q\}$. For instance, $\beta_{[i,j]}=\{(x,y) \in \beta : i \le x,y \le j\}$.

Let $\alpha$ be an infinite word over the alphabet $\{0,1,2,3\}$, $\beta$ be a bond set and $\gamma$ be an infinite binary word. We refer to the three objects combined as a \emph{$\delta$-triple}, denoted $\delta=(\alpha,\beta,\gamma)$. 

We define an infinite graph $\PPP^\delta$ with vertices $V(\PPP)$ and with edges defined by $\delta$ as follows:
\begin{enumerate}[label=(\alph*)]
\item \emph{$\alpha$-edges} between consecutive columns determined by the letters of the word $\alpha$. For each $i=1,2,\dots$, the edges between $C_i$ and $C_{i+1}$ are:

	\begin{enumerate}[label=(\roman*)]
	\item $\{(v_{i,j}, v_{i+1,j}) : j \in \mathbb{N}\}$ if $\alpha_i=0$   (i.e.\ a matching);
	\item $\{(v_{i,j}, v_{i+1,k}) : j \neq k; j,k \in \mathbb{N}\}$ if $\alpha_i=1$  (i.e.\ the bipartite complement \footnote{The \emph{bipartite complement} $\hat{G}$ of a bipartite graph $G$ has the same independent vertex sets $V_1$ and $V_2$ as $G$ where vertices $v_1 \in V_1$ and $v_2 \in V_2$ are adjacent in $\hat{G}$ if and only if they are not adjacent in $G$.} of a matching);
	\item $\{(v_{i,j}, v_{i+1,k}) : j \ge k; j,k \in \mathbb{N}\}$ if $\alpha_i=2$;
	\item $\{(v_{i,j}, v_{i+1,k}) : j < k; j,k \in \mathbb{N}\}$ if $\alpha_i=3$ (i.e.\ the bipartite complement of a $2$).
	\end{enumerate}
 
\item \emph{$\beta$-edges} defined by the bond-set $\beta$ such that $v_{i,x} \sim v_{j,y}$ for all $x,y \in \mathbb{N}$ when $(i,j) \in \beta$ (i.e. a complete bipartite graph between $C_i$ and $C_j$), and
\item \emph{$\gamma$-edges} defined by the letters of the binary word $\gamma$ such that for any $j,k \in \mathbb{N}$ we have $v_{i,j} \sim v_{i,k}$ if and only if $\gamma_i=1$ (i.e. $C_i$ forms a clique if $\gamma_i=1$ and an independent set if $\gamma_i=0$). 
\end{enumerate}
  
The hereditary graph class $\GGG^\delta$ is the set of all finite induced subgraphs of $\PPP^\delta$. 

Any graph $G \in \GGG^\delta$ can be witnessed by an embedding $\phi(G)$ into the infinite graph $\PPP^\delta$. To simplify the presentation we will associate $G$ with a particular embedding in $\PPP^\delta$ depending on the context. We will be especially interested in the induced subgraphs of $G$ that occur in consecutive columns: in particular, an \emph{$\alpha_j$-link} is the induced subgraph of $G$ on the vertices of $G\cap C_{[j,j+1]}$, and will be denoted by $G_{[j,j+1]}$. More generally, an induced subgraph of $G$ on the vertices of $G\cap C_{[j,k]}$ will be denoted $G_{[j,k]}$. 

For $k \ge 2$ we denote the triple $\delta_{[j,j+k-1]}=(\alpha_{[j,j+k-2]} ; \beta_{[j,j+k-1]} ; \gamma_{[j,j+k-1]})$ as a \emph{$k$-factor} of $\delta$. Thus for a graph $G \in \GGG^\delta$  with a particular embedding in $\PPP^\delta$, the induced subgraph $G_{[j,j+k-1]}$ has edges defined by the $k$-factor $\delta_{[j,j+k-1]}$.

We say that two $k$-factors $\delta_{[x,x+k]}$ and $\delta_{[y,y+k]}$ are the same if 
\begin{enumerate}[label=(\roman*)]
	\item for all $i \in [0,k-1]$, $\alpha_{x+i}=\alpha_{y+i}$, and
	\item for all $i,j \in [0,k], (x+i,x+j) \in \beta$ if and only 
	if $(y+i,y+j) \in \beta$, and
	\item for all $i \in [0,k]$, $\gamma_{x+i}=\gamma_{y+i}$.
\end{enumerate} 

We say that a $\delta$-triple is \emph{recurrent} if  every $k$-factor occurs in it infinitely many times. We say that $\delta$ is \emph{almost periodic} if for each $k$-factor $\delta_{[j,k]}$ of $\delta$ there exists a constant $\mathcal{L}(\delta_{[j,k]})$ such that every factor of $\delta$ of length $\mathcal{L}(\delta_{[j,k]})$ contains $\delta_{[j,k]}$ as a factor.

A \emph{couple set} $P$ is a subset of $\mathbb{N}$ such that if $x,y \in P$ then $|x-y|>2$. Such a set is used to identify sets of links that have no $\alpha$-edges between them.  We say that a pair $(x, y)$ of elements of $P$ is \emph{$\beta$-dense} if both $(x,y+1)$ and $(x+1,y)$ are in $\beta$ and they are \emph{$\beta$-sparse} when neither of these bonds is in $\beta$.\label{dense-sparse} 

We say the bond-set $\beta$ is \emph{sparse} in $P$ if every pair from $P$ is $\beta$-sparse and is \emph{not sparse} in $P$ if there are no $\beta$-sparse pairs in $P$. Likewise, $\beta$ is  \emph{dense} in $P$ if every pair from $P$ is $\beta$-dense and is \emph{not dense} in $P$ if there are no $\beta$-dense pairs in $P$. Clearly it is possible for two elements from $P$ to be neither $\beta$-sparse nor $\beta$-dense (i.e. when only one of the required bonds is in $\beta$). These ideas are used to identify matched and unmatched distinguished pairings (see Lemmas \ref{lem:beta-not-dense} 
and \ref{lem:beta-not-sparse}).

%
%
%
%
%
%
\subsection{Clique-width and linear clique-width}\label{sect_cwd}

\emph{Clique-width} is a graph width parameter introduced by Courcelle, Engelfriet and Rozenberg in the 1990s~\cite{courcelle:handle-rewritin:}. The clique-width of a graph is denoted $\cwd(G)$ and is defined as the minimum number of labels needed to construct $G$ by means of the following four graph operations:

\begin{enumerate}[label=(\alph*)]
\item creation of a new vertex $v$ with label $i$ (denoted $i(v)$),
\item adding an edge between every vertex labelled $i$ and every vertex labelled $j$ for distinct $i$ and $j$ (denoted $\eta_{i,j}$),
\item giving all vertices labelled $i$ the label $j$ (denoted $\rho_{i \rightarrow j}$), and
\item taking the disjoint union of two previously-constructed labelled graphs $G$ and $H$, one of which may be empty (denoted $G \oplus H$).
\end{enumerate}

The \emph{linear clique-width} of a graph $G$ denoted $lcw(G)$ is the minimum number of labels required to construct $G$ by means of  four operations, being $(a),(b),(c)$ above plus '$(d)$ taking the disjoint union of two previously-constructed labelled graphs $G$ and $H$, one of which is a single labelled vertex $v$ (denoted $G \oplus v$) or no vertex (denoted $G \oplus \emptyset$)'. 

Every graph can be defined by an algebraic expression $\tau$ using the four operations above, which we will refer to as a \emph{(linear) clique-width expression}. This expression is called a \emph{$k$-expression} if it uses $k$ different labels.

Alternatively, any clique-width expression $\tau$ defining $G$ can be represented as a rooted binary tree, $\tree(\tau)$, whose leaves correspond to the operations of vertex creation, the internal nodes correspond to the $\oplus$-operation, and the root is associated with $G$.  The operations $\eta$ and $\rho$ are assigned in the appropriate sequence along the respective edges of  $\tree(\tau)$. The tree is binary since each $\oplus$-operation brings together at most two previously constructed graphs. Also, it can be observed that an $\oplus$-vertex represents a subgraph of $G$ but not usually an induced subgraph since there may still be edges to be created by $\eta$ operations.

In the case of a linear clique-width expression the tree becomes a \emph{caterpillar tree}, that is, a tree that becomes a path after the removal of the leaves.

Clearly from the definition, $lcw(G) \ge cwd(G)$. Hence, a graph class of unbounded clique-width is also a class of unbounded linear clique-width. Likewise, a class with bounded linear clique-width is also a class of bounded clique-width. 

A hereditary class of graphs $\CCC$ is \emph{minimal of unbounded clique-width} or just \emph{minimal} if every proper subclass $\DDD \subsetneq \CCC$ has bounded clique-width. In other words, if $\CCC=Free(H_1,H_2,\dots)$ then it is minimal if any proper subclass $\DDD$ formed by adding just one more forbidden graph  has  bounded clique-width. Thus, if $\CCC$ has unbounded clique-width but $\CCC \cap Free(H)$ has bounded linear clique-width for any non-trivial graph $H$, then $\CCC$ is minimal of unbounded clique-width \underline{and} minimal of unbounded linear clique-width.

%
%
%
%
%
%
\section{\texorpdfstring{$\GGG^\delta$}{\GGG{}} graph classes with unbounded clique-width} \label{Sect:Unbounded}

This section identifies which hereditary classes $\mathcal{G}^\delta$ have unbounded clique-width. We prove this is determined by a neighbourhood parameter $\NNN^\delta$ derived from a graph induced on any two rows of the graph $\PPP^\delta$. We show that $\GGG^\delta$ has unbounded clique-width if and only if  $\NNN^\delta$ is unbounded (Theorem \ref{thm-0123unbound}).

%
%
%
%
\subsection{The two-row graph and \texorpdfstring{$\mathcal{N}^\delta$}{ }} \label{two-row}

We will show that the boundedness of clique-width for a graph class $\GGG^\delta$ is determined by the adjacencies between the first two rows of  $\PPP^\delta$ (it could, in fact, be any two rows). Hence, the following graph is useful.

A \emph{two-row graph} $T^{\delta}(Q)=(V,E)$ is the subgraph of $\PPP^\delta$ induced on the vertices $V=R_1(Q) \cup R_2(Q)$ where $R_1(Q)=\{v_{i,1}: i \in Q\}$ and $R_2(Q)=\{v_{j,2}: j \in Q \}$ for finite subset $Q \subseteq\mathbb{N}$.

We define the parameter 
$\mathcal{N}^\delta(Q)= \mu(T^\delta(Q),R_1(Q))$.

\begin{lemma}\label{lem:Nbound}
For any fixed $j \in \mathbb{N}$, $\mathcal{N}^\delta([1,n])$ is bounded as $n \rightarrow \infty$ if and only if $\mathcal{N}^\delta([j,n])$ is bounded as $n \rightarrow \infty$.
\end{lemma}
\begin{proof}
It is easy to see that if there exists $N$ such that $\mathcal{N}^\delta([1,n])<N$ for all $n \in \mathbb{N}$ then $\mathcal{N}^\delta([j,n]) < N$ for all $n \in \mathbb{N}$.

On the other hand, if $\mathcal{N}^\delta([j,n])<N$ then $\mathcal{N}^\delta([j-1,n]) < 2N+1$ since by adding the extra column each 'old' equivalence class could at most be split in two and there is one new vertex in each row. By induction we have
$\mathcal{N}^\delta([1,n]) < 2^jN +\sum_{i=0}^{j-1} 2^i$
for all $n \in \mathbb{N}$.
\end{proof}

We will say \emph{$\mathcal{N}^\delta$ is unbounded} if $\mathcal{N}^\delta([j,n])$ is unbounded as $n \rightarrow \infty$ for some fixed $j \in \mathbb{N}$. In many cases it is simple to check that $\mathcal{N}^\delta$ is unbounded -- e.g. the following $\delta$-triples have unbounded $\mathcal{N}^\delta$:
\[(1^\infty, \emptyset, 0^\infty), (2^\infty, \emptyset, 0^\infty), (3^\infty, \emptyset, 0^\infty), (0^\infty, \emptyset, 1^\infty)\]
In Lemma \ref{23-N-unbound} we show that  $\mathcal{N}^\delta$ is unbounded whenever $\alpha$ contains an infinite number of $2$s or $3$s.

%
%
%
%
\subsection{Clique-width expression and colour partition for an  \texorpdfstring{$n\times n$}{n{}n} square graph}

We will denote $H^\delta_{i,j}(m,n)$ as the $m (cols) \times n (rows)$ induced subgraph of $\PPP^\delta$ formed from the rectangular grid of vertices $\{v_{x,y} : x \in [i,i+m-1], y \in [j,j+n-1]\}$. See Figure \ref{Hgraph}.

\begin{figure}[ht]
{\centering
\begin{tikzpicture}[scale=1]
	
	\mnnodearray{9}{6}
		\foreach \col in {1,5}
			\horiz{\col}{6};
		\foreach \col in {2,6}
			\downnohoriz{\col}{6};
		\foreach \col in {2,6}
			\upnohoriz{\col}{6};	
		\foreach \col in {3,7}
			\downhoriz{\col}{6};
		\foreach \col in {4,8}
			\upnohoriz{\col}{6};
	
\end{tikzpicture}\par }

\caption{$H^\delta_{1,1}(9,6)$ where $\alpha=01230123\cdots$ ($\beta$ and $\gamma$ edges not shown)}
	\label{Hgraph}

\end{figure}
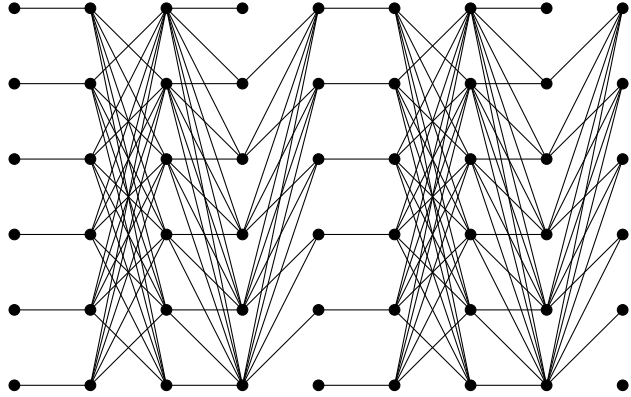


We can calculate a lower bound for the clique-width of the $n \times n$ square graph $H^\delta_{j,1}(n,n)$ (shortened to $H(n,n)$ when $\delta$, $j$ and $1$ are clearly implied), by demonstrating a minimum number of labels needed to construct it using the allowed four graph operations, as follows.

Let $\tau$ be a clique-width expression defining $H(n,n)$ and $\tree(\tau)$ the rooted tree representing $\tau$. The subtree of $\tree(\tau)$ rooted at a node $\oplus$ corresponds to a subgraph of $H(n,n)$.  We can give this node a label, say $a$,  so that $\oplus_a$ is the root and $H_a$ the corresponding subgraph of $H(n,n)$. 

We denote by $\oplus_{red}$ and $\oplus_{blue}$ the two children of $\oplus_a$ in $\tree(\tau)$. Let us colour the vertices of $H_{red}$ and $H_{blue}$ red and blue, respectively, and all the other vertices in $H(n,n)$ white. Let $\Colour(v)$ denote the colour of a vertex $v \in H(n,n)$ as described above, and $\Label(v)$ denote the label of vertex $v$ (if any) at node $\oplus_a$. (If $v$ is white it is a vertex of $H(n,n)$ not in subgraph $H_a$ and therefore it has either been created in a branch of $\tree(\tau)$ not yet connected to node $\oplus_a$, or has not yet been created, in which case we say $\Label(v)=\epsilon$).

Our identification of a minimum number of labels needed to construct $H(n,n)$ relies on the following observation regarding this vertex colour partition.

\begin{obs}\label{obs1}
Suppose $u_1$, $u_2$, $w$ are three vertices in $H(n,n)$ such that $u_1$ and $u_2$ are non-white, $ u_1 \sim w$ but $u_2 \not\sim w$, and $\Colour(w)\neq \Colour(u_1)$. Then $u_1$ and $u_2$ must have different labels at node $\oplus_a$.
\end{obs}

This is true because the edge $u_1w$ still needs to be created, whilst respecting the non-adjacency of $u_2$ and $w$. We now focus on sets of blue and sets of nonblue vertices (Equally, we could have chosen red-nonred). Observation \ref{obs1} leads to the following key lemma which is the basis of much which follows.

\begin{lemma}\label{labels}
For graph $H(n,n)$ let $U$ and $W$ be two disjoint vertex sets with induced subgraph $H=H(n,n)[U \cup W]$ such that $\mu(H,U)=r$. Then if the vertex colouring described above gives $\Colour(u)=$ blue for all $u \in U$ and $\Colour(w)\neq$ blue for all $w \in W$ then the clique-width expression $\tau$ requires at least $r$ labels at node $\oplus_a$.
\end{lemma}
\begin{proof}
Choose one representative vertex from each equivalence class in $U$. For any two such representatives $u_1$ and $u_2$ there must exist a $w$ in $W$ such that $ u_1 \sim w$ but $u_2 \not\sim w$ (or vice versa). By Observation \ref{obs1} $u_1$ and $u_2$ must have different labels at node $\oplus_a$. This applies to any pair of representatives $u_1, u_2$ and hence all $r$ such vertices must have distinct labels.
\end{proof}


Note that from Proposition \ref{prop-pair} a distinguished pairing gives us  the sets $U$ and $W$ required for Lemma \ref {labels}. The following lemmas identify structures in $H(n,n)$ that give us these distinguished pairings. 

We denote by $H_{[y,y+1]}$ the $\alpha_y$-link $H(n,n) \cap C_{[y,y+1]}$ where $y \in [j,j+n-2]$. We will refer to a (adjacent or non-adjacent) \emph{blue-nonblue pair} to mean two vertices, one of which is coloured blue and one non-blue, such that they are in consecutive columns, where the blue vertex could be to the left or the right of the nonblue vertex. If we have a set of such pairs with the blue vertex on the same side (i.e. on the left or right) then we say the pairs in the set have the same \emph{polarity}.

\begin{lemma}\label{lem-blue-blue}
Suppose that $H_{[y,y+1]}$ contains a horizontal pair $(b_1, b_2)$ of blue vertices and at least one nonblue vertex $n_1$, $n_2$ in each column, but not on the top or bottom row (see Figure \ref{fig:blue-blue}).
\begin{enumerate}[label=(\alph*)]
\item If $\alpha_y \in \{0,2,3\}$ then $H_{[y,y+1]}$ contains a non-adjacent blue-nonblue pair.
\item If $\alpha_y \in \{1,2,3\}$ then $H_{[y,y+1]}$ contains an adjacent blue-nonblue pair.
\end{enumerate}
\end{lemma}
\begin{proof}
If $\alpha_y = 0$ then both $(b_1,n_1)$ and $(b_2,n_2)$ form a non-adjacent blue-nonblue pair (Figure \ref{fig:blue-blue} A). If $\alpha_y = 1$ then both $(b_1,n_1)$ and $(b_2,n_2)$ form an adjacent blue-nonblue pair (Figure \ref{fig:blue-blue} B).

If $\alpha_y \in \{2,3\}$ and the nonblue vertices $n_1$ and $n_2$ in each column are either both above or both below the horizontal blue pair $(b_1,b_2)$ then it can be seen that one of the pairs $(b_1,n_1)$ or $(b_2,n_2)$ forms an adjacent blue-nonblue pair and the other forms a non-adjacent blue-nonblue pair (Figure \ref{fig:blue-blue} C). If the nonblue vertices in each column are either side of the blue pair (one above and one below) then the pairs $(b_1,n_1)$ and $(b_2,n_2)$ will both be adjacent (or non-adjacent) blue-nonblue pairs (See Figure \ref{fig:blue-blue} D). In this case we need to appeal to a $5$-th vertex $s$ which will form a non-adjacent (or adjacent) set with either $n_1$ or $b_2$ depending on its colour. Thus we always have both a non-adjacent and adjacent blue-non-blue pair when $\alpha_y \in \{2,3\}$.
\end{proof}

\begin{figure}\centering
\begin{tikzpicture}[scale=0.5,
vertex2/.style={circle,draw,minimum size=6,fill=blue},
vertex3/.style={circle,draw,minimum size=6,fill=yellow},
vertex4/.style={circle,draw=white,minimum size=6,fill=white},
vertex5/.style={circle,draw=gray,minimum size=6,fill=gray},]
	
	\begin{scope}[shift={(0,0)}]	
	
	\node (1)[vertex5] at (-10,2) {};
	\node (2)[vertex5] at (-8,2) {};
	\node (3)[vertex5] at (-10,4) {};
	\node (4)[vertex5] at (-8,4) {};
	\node (5)[label={left:$b_1$}][vertex2] at (-10,6) {};
	\node (6)[label={right:$b_2$}][vertex2] at (-8,6) {};
	\node (7)[label={left:$n_2$}][vertex3] at (-10,8) {};
	\node (8)[vertex5] at (-8,8) {};
	\node (9)[vertex5] at (-10,10) {};
	\node (10)[label={right:$n_1$}][vertex3] at (-8,10) {};
	
	\draw [gray](1) -- (2); 
	\draw [gray](3) -- (4);
	\draw [gray](5) -- (6);
	\draw [gray](7) -- (8);
	\draw [gray](9) -- (10);
	
	\node (11)[vertex5] at (-4,2) {};	
	\node (12)[vertex5] at (-2,2) {};
	\node (13)[vertex5] at (-4,4) {};
	\node (14)[vertex5] at (-2,4) {};
	\node (15)[label={left:$b_1$}][vertex2] at (-4,6) {};
	\node (16)[label={right:$b_2$}][vertex2] at (-2,6) {};
	\node (17)[label={left:$n_2$}][vertex3] at (-4,8) {};
	\node (18)[vertex5] at (-2,8) {};
	\node (19)[vertex5] at (-4,10) {};
	\node (20)[label={right:$n_1$}][vertex3] at (-2,10) {};	
	
	\draw [gray](11) -- (14); 
	\draw [gray](11) -- (16); 
	\draw [gray](11) -- (18);
	\draw [gray](11) -- (20);
	\draw [gray](13) -- (12); 
	\draw [gray](13) -- (16); 
	\draw [gray](13) -- (18);
	\draw [gray](13) -- (20);
	\draw [gray](15) -- (12); 
	\draw [gray](15) -- (14); 
	\draw [gray](15) -- (18);
	\draw (15) -- (20);
	\draw [gray](17) -- (12); 
	\draw [gray](17) -- (14); 
	\draw (17) -- (16);
	\draw [gray](17) -- (20);
	\draw [gray](19) -- (12); 
	\draw [gray](19) -- (14); 
	\draw [gray](19) -- (16);
	\draw [gray](19) -- (18);
	
	\node (21)[vertex5] at (2,2) {};	
	\node (22)[vertex5] at (4,2) {};
	\node (23)[vertex5] at (2,4) {};
	\node (24)[vertex5] at (4,4) {};
	\node (25)[label={left:$b_1$}][vertex2] at (2,6) {};
	\node (26)[label={right:$b_2$}][vertex2] at (4,6) {};
	\node (27)[label={left:$n_2$}][vertex3] at (2,8) {};
	\node (28)[vertex5] at (4,8) {};
	\node (29)[vertex5] at (2,10) {};
	\node (30)[label={right:$n_1$}][vertex3] at (4,10) {};
	
	\draw [gray](22) -- (21); 
	\draw [gray](22) -- (23); 
	\draw [gray](22) -- (25);
	\draw [gray](22) -- (27);
	\draw [gray](22) -- (29);
	\draw [gray](24) -- (23);
	\draw [gray](24) -- (25);
	\draw [gray](24) -- (27);
	\draw [gray](24) -- (29);
	\draw [gray](26) -- (25);
	\draw (26) -- (27);
	\draw [gray](26) -- (29);
	\draw [gray](28) -- (27);
	\draw [gray](28) -- (29);
	\draw [gray](30) -- (29);	
	
	\node (31)[label={left:$n_2$}][vertex3] at (8,2) {};
	\node (32)[vertex5] at (10,2) {};
	\node (33)[vertex5] at (8,4) {};
	\node (34)[vertex5] at (10,4) {};
	\node (35)[label={left:$b_1$}][vertex2] at (8,6) {};
	\node (36)[label={right:$b_2$}][vertex2] at (10,6) {};
	\node (37)[vertex5] at (8,8) {};
	\node (38)[label={right:$n_1$}][vertex3] at (10,8) {};
	\node (39)[label={left:$s$}][vertex5] at (8,10) {};
	\node (40)[vertex5] at (10,10) {};

	\draw [gray](31) -- (34);
	\draw (31) -- (36); 
	\draw [gray](31) -- (38);
	\draw [gray](31) -- (40);
	\draw [gray](33) -- (36);
	\draw [gray](33) -- (38);
	\draw [gray](33) -- (40);
	\draw (35) -- (38);
	\draw [gray](35) -- (40);
	\draw [gray](37) -- (40);

\node (101)[label={below:A. $0$-link}] [vertex4] at (-9,1) {};
\node (102)[label={below:B. $1$-link}] [vertex4] at (-3,1) {};	
\node (103)[label={below:C. $2$-link}] [vertex4] at (3,1) {};		
\node (104)[label={below:D. $3$-link}] [vertex4] at (9,1) {};
	
	\end{scope}
		
\end{tikzpicture}

\caption{Horizontal blue-blue pair in $H_{[y,y+1]}$ (nonblue vertices in yellow)}
	\label{fig:blue-blue}

\end{figure}
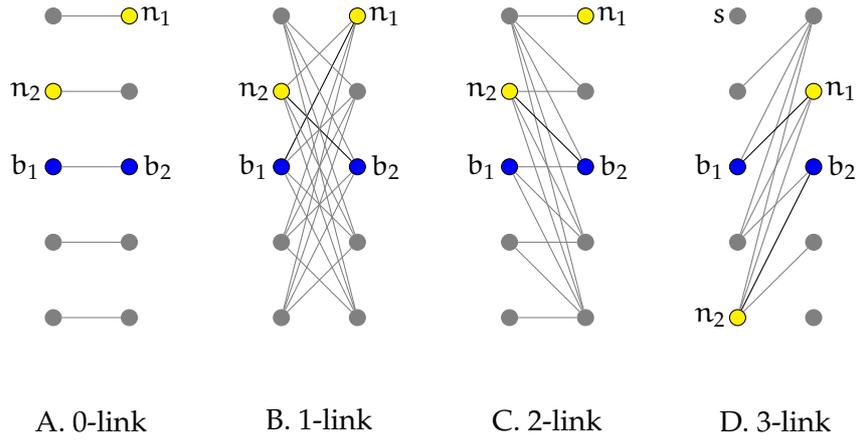


\begin{lemma}\label{lem-blue-nonblue2}
Suppose $H_{[y,y+1]}$ contains a horizontal blue-nonblue pair of vertices $(b_1,n_1)$, not the top or bottom row, and at least one nonblue vertex $n_2$ in the same column as $b_1$. Then $H_{[y,y+1]}$ contains both an adjacent and a non-adjacent blue-nonblue pair of vertices, irrespective of the value of $\alpha_y$ (see Figure \ref{fig:blue-nonblue2}).
\end{lemma}
\begin{proof}
If $\alpha_y \in \{0,2\}$ then the horizontal blue-nonblue pair $(b_1,n_1)$ is adjacent, and given a nonblue vertex $n_2$ in the same column as $b_1$, we can find a vertex $s$ in the same column as $n_1$ that forms a non-adjacent pairing with either $b_1$ or $n_2$ depending on its colour (See Figure \ref{fig:blue-nonblue2} A and C). If $\alpha_y \in \{1,3\}$ then the horizontal blue-nonblue pair $(b_1,n_1)$ is non-adjacent, and given a nonblue vertex $n_2$ in the same column as $b_1$, we can find a vertex $s$ in the same column as $n_1$ that forms an adjacent pairing with either $b_1$ or $n_2$ depending on its colour (See Figure \ref{fig:blue-nonblue2} B and D).
\end{proof}

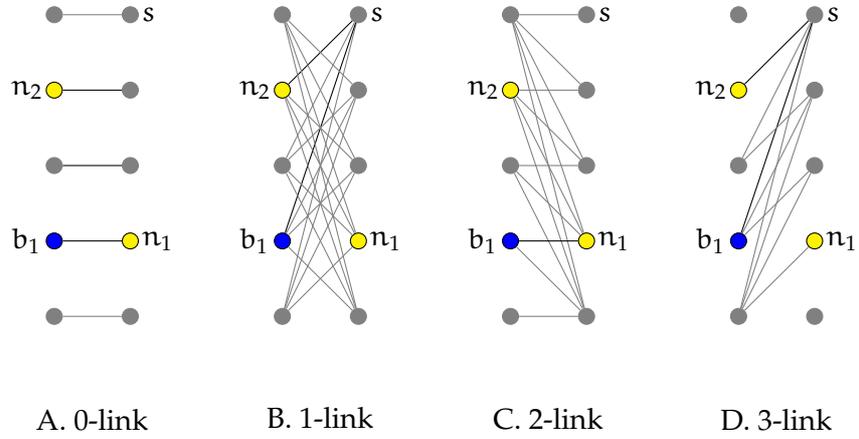
\begin{figure}\centering
\begin{tikzpicture}[scale=0.5,
vertex2/.style={circle,draw,minimum size=6,fill=blue},
vertex3/.style={circle,draw,minimum size=6,fill=yellow},
vertex4/.style={circle,draw=white,minimum size=6,fill=white},
vertex5/.style={circle,draw=gray,minimum size=6,fill=gray},]
		
	\begin{scope}[shift={(0,0)}]		
	
	\node (1)[vertex5] at (-10,2) {};
	\node (2)[vertex5] at (-8,2) {};
	\node (3)[label={left:$b_1$}][vertex2] at (-10,4) {};
	\node (4)[label={right:$n_1$}][vertex3] at (-8,4) {};
	\node (5)[vertex5] at (-10,6) {};
	\node (6)[vertex5] at (-8,6) {};
	\node (7)[label={left:$n_2$}][vertex3] at (-10,8) {};
	\node (8)[vertex5] at (-8,8) {};
	\node (9)[vertex5] at (-10,10) {};
	\node (10)[label={right:$s$}][vertex5] at (-8,10) {};
	
	\draw [gray](1) -- (2); 
	\draw (3) -- (4);
	\draw (5) -- (6);
	\draw (7) -- (8);
	\draw [gray](9) -- (10);
	
	\node (11)[vertex5] at (-4,2) {};	
	\node (12)[vertex5] at (-2,2) {};
	\node (13)[label={left:$b_1$}][vertex2] at (-4,4) {};
	\node (14)[label={right:$n_1$}][vertex3] at (-2,4) {};
	\node (15)[vertex5] at (-4,6) {};
	\node (16)[vertex5] at (-2,6) {};
	\node (17)[label={left:$n_2$}][vertex3] at (-4,8) {};
	\node (18)[vertex5] at (-2,8) {};
	\node (19)[vertex5] at (-4,10) {};
	\node (20)[label={right:$s$}][vertex5] at (-2,10) {};	
	
	\draw [gray](11) -- (14); 
	\draw [gray](11) -- (16); 
	\draw [gray](11) -- (18);
	\draw [gray](11) -- (20);
	\draw [gray](13) -- (12); 
	\draw [gray](13) -- (16); 
	\draw [gray](13) -- (18);
	\draw (13) -- (20);
	\draw [gray](15) -- (12); 
	\draw [gray](15) -- (14); 
	\draw [gray](15) -- (18);
	\draw [gray](15) -- (20);
	\draw [gray](17) -- (12); 
	\draw [gray](17) -- (14); 
	\draw [gray](17) -- (16);
	\draw (17) -- (20);
	\draw [gray](19) -- (12); 
	\draw [gray](19) -- (14); 
	\draw [gray](19) -- (16);
	\draw [gray](19) -- (18);
	
	\node (21)[vertex5] at (2,2) {};	
	\node (22)[vertex5] at (4,2) {};
	\node (23)[label={left:$b_1$}][vertex2] at (2,4) {};
	\node (24)[label={right:$n_1$}][vertex3] at (4,4) {};
	\node (25)[vertex5] at (2,6) {};
	\node (26)[vertex5] at (4,6) {};
	\node (27)[label={left:$n_2$}][vertex3] at (2,8) {};
	\node (28)[vertex5] at (4,8) {};
	\node (29)[vertex5] at (2,10) {};
	\node (30)[label={right:$s$}][vertex5] at (4,10) {};
	
	\draw [gray](22) -- (21); 
	\draw [gray](22) -- (23); 
	\draw [gray](22) -- (25);
	\draw [gray](22) -- (27);
	\draw [gray](22) -- (29);
	\draw (24) -- (23);
	\draw [gray](24) -- (25);
	\draw [gray](24) -- (27);
	\draw [gray](24) -- (29);
	\draw [gray](26) -- (25);
	\draw [gray](26) -- (27);
	\draw [gray](26) -- (29);
	\draw [gray](28) -- (27);
	\draw [gray](28) -- (29);
	\draw [gray](30) -- (29);	
	
	\node (31)[vertex5] at (8,2) {};
	\node (32)[vertex5] at (10,2) {};
	\node (33)[label={left:$b_1$}][vertex2] at (8,4) {};
	\node (34)[label={right:$n_1$}][vertex3] at (10,4) {};
	\node (35)[vertex5] at (8,6) {};
	\node (36)[vertex5] at (10,6) {};
	\node (37)[label={left:$n_2$}][vertex3] at (8,8) {};
	\node (38)[vertex5] at (10,8) {};
	\node (39)[vertex5] at (8,10) {};
	\node (40)[label={right:$s$}][vertex5] at (10,10) {};

	\draw [gray](31) -- (34);
	\draw [gray](31) -- (36); 
	\draw [gray](31) -- (38);
	\draw [gray](31) -- (40);
	\draw [gray](33) -- (36);
	\draw [gray](33) -- (38);
	\draw (33) -- (40);
	\draw [gray](35) -- (38);
	\draw [gray](35) -- (40);
	\draw (37) -- (40);

\node (101)[label={below:A. $0$-link}] [vertex4] at (-9,1) {};
\node (102)[label={below:B. $1$-link}] [vertex4] at (-3,1) {};	
\node (103)[label={below:C. $2$-link}] [vertex4] at (3,1) {};		
\node (104)[label={below:D. $3$-link}] [vertex4] at (9,1) {};
	
	\end{scope}

\end{tikzpicture}

\caption{Horizontal blue-nonblue pair in $H_{[y,y+1]}$ (nonblue vertices in yellow)}
	\label{fig:blue-nonblue2}

\end{figure}


\begin{lemma}\label{lem-blue-nonblue}
Suppose $H_{[y,y+1]}$ contains $r \ge 3$ horizontal blue-nonblue pairs of vertices $(b_i,n_i)$, $i=1,\dots,r$, with the same polarity (see Figure \ref{fig:blue-nonblue}). Then, irrespective of the value of $\alpha_y$, it contains 
\begin{enumerate}[label=(\alph*)]
\item a matched distinguished pairing $\{U,W\}$  of size $r-1$ such that $\Colour(u)=$ blue for all $u \in U$ and $\Colour(w)\neq$ blue for all $w \in W$, and 
\item an unmatched distinguished pairing $\{U^\prime,W^\prime\}$ of size $r-1$ such that $\Colour(u^\prime)=$ blue for all $u^\prime \in U^\prime$ and $\Colour(w^\prime)\neq$ blue for all $w^\prime \in W^\prime$.  
\end{enumerate}
\end{lemma}
\begin{proof}
This is easily observable from Figure \ref{fig:blue-nonblue} for $r=3$. If we set $U=\{b_1,b_2\}$, $W=\{n_1,n_2\}$, $U^\prime=\{b_2,b_3\}$ and $W^\prime=\{n_1,n_2\}$ then one of $\{U,W\}$ and $\{U^\prime,W^\prime\}$ is a matched distinguished pairing of size $2$ and the other is an unmatched distinguished pairing of size $2$, irrespective of the value of $\alpha_y$. Simple induction establishes this for all $r \ge 3$. 
\end{proof}

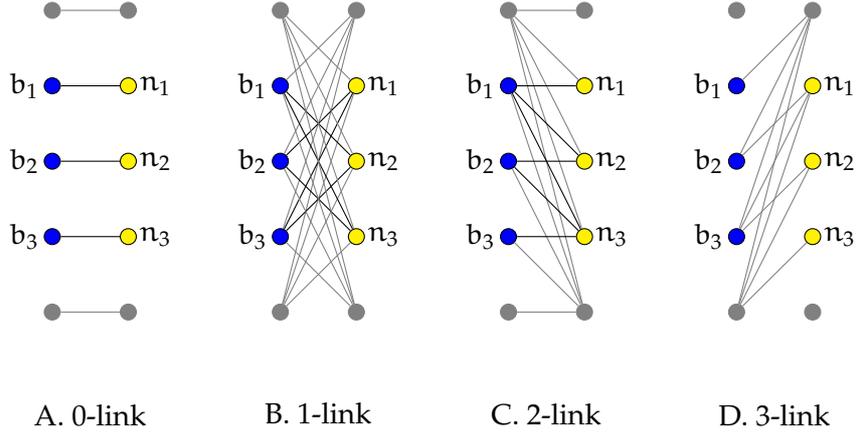
\begin{figure}\centering
\begin{tikzpicture}[scale=0.5,
vertex2/.style={circle,draw,minimum size=6,fill=blue},
vertex3/.style={circle,draw,minimum size=6,fill=yellow},
vertex4/.style={circle,draw=white,minimum size=6,fill=white},
vertex5/.style={circle,draw=gray,minimum size=6,fill=gray},]
		
	\begin{scope}[shift={(0,0)}]		
	
	\node (1)[vertex5] at (-10,2) {};
	\node (2)[vertex5] at (-8,2) {};
	\node (3)[label={left:$b_3$}][vertex2] at (-10,4) {};
	\node (4)[label={right:$n_3$}][vertex3] at (-8,4) {};
	\node (5)[label={left:$b_2$}][vertex2] at (-10,6) {};
	\node (6)[label={right:$n_2$}][vertex3] at (-8,6) {};
	\node (7)[label={left:$b_1$}][vertex2] at (-10,8) {};
	\node (8)[label={right:$n_1$}][vertex3] at (-8,8) {};
	\node (9)[vertex5] at (-10,10) {};
	\node (10)[vertex5] at (-8,10) {};
	
	\draw [gray](1) -- (2); 
	\draw (3) -- (4);
	\draw (5) -- (6);
	\draw (7) -- (8);
	\draw [gray](9) -- (10);
	
	\node (11)[vertex5] at (-4,2) {};	
	\node (12)[vertex5] at (-2,2) {};
	\node (13)[label={left:$b_3$}][vertex2] at (-4,4) {};
	\node (14)[label={right:$n_3$}][vertex3] at (-2,4) {};
	\node (15)[label={left:$b_2$}][vertex2] at (-4,6) {};
	\node (16)[label={right:$n_2$}][vertex3] at (-2,6) {};
	\node (17)[label={left:$b_1$}][vertex2] at (-4,8) {};
	\node (18)[label={right:$n_1$}][vertex3] at (-2,8) {};
	\node (19)[vertex5] at (-4,10) {};
	\node (20)[vertex5] at (-2,10) {};	
	
	\draw [gray](11) -- (14); 
	\draw [gray](11) -- (16); 
	\draw [gray](11) -- (18);
	\draw [gray](11) -- (20);
	\draw [gray](13) -- (12); 
	\draw (13) -- (16); 
	\draw (13) -- (18);
	\draw [gray](13) -- (20);
	\draw [gray](15) -- (12); 
	\draw (15) -- (14); 
	\draw (15) -- (18);
	\draw [gray](15) -- (20);
	\draw [gray](17) -- (12); 
	\draw (17) -- (14); 
	\draw (17) -- (16);
	\draw [gray](17) -- (20);
	\draw [gray](19) -- (12); 
	\draw [gray](19) -- (14); 
	\draw [gray](19) -- (16);
	\draw [gray](19) -- (18);
	
	\node (21)[vertex5] at (2,2) {};	
	\node (22)[vertex5] at (4,2) {};
	\node (23)[label={left:$b_3$}][vertex2] at (2,4) {};
	\node (24)[label={right:$n_3$}][vertex3] at (4,4) {};
	\node (25)[label={left:$b_2$}][vertex2] at (2,6) {};
	\node (26)[label={right:$n_2$}][vertex3] at (4,6) {};
	\node (27)[label={left:$b_1$}][vertex2] at (2,8) {};
	\node (28)[label={right:$n_1$}][vertex3] at (4,8) {};
	\node (29)[vertex5] at (2,10) {};
	\node (30)[vertex5] at (4,10) {};
	
	\draw [gray](22) -- (21); 
	\draw [gray](22) -- (23); 
	\draw [gray](22) -- (25);
	\draw [gray](22) -- (27);
	\draw [gray](22) -- (29);
	\draw (24) -- (23);
	\draw (24) -- (25);
	\draw (24) -- (27);
	\draw [gray](24) -- (29);
	\draw (26) -- (25);
	\draw (26) -- (27);
	\draw [gray](26) -- (29);
	\draw (28) -- (27);
	\draw [gray](28) -- (29);
	\draw [gray](30) -- (29);	
	
	\node (31)[vertex5] at (8,2) {};
	\node (32)[vertex5] at (10,2) {};
	\node (33)[label={left:$b_3$}][vertex2] at (8,4) {};
	\node (34)[label={right:$n_3$}][vertex3] at (10,4) {};
	\node (35)[label={left:$b_2$}][vertex2] at (8,6) {};
	\node (36)[label={right:$n_2$}][vertex3] at (10,6) {};
	\node (37)[label={left:$b_1$}][vertex2] at (8,8) {};
	\node (38)[label={right:$n_1$}][vertex3] at (10,8) {};
	\node (39)[vertex5] at (8,10) {};
	\node (40)[vertex5] at (10,10) {};

	\draw [gray](31) -- (34);
	\draw [gray](31) -- (36); 
	\draw [gray](31) -- (38);
	\draw [gray](31) -- (40);
	\draw [gray](33) -- (36);
	\draw [gray](33) -- (38);
	\draw [gray](33) -- (40);
	\draw [gray](35) -- (38);
	\draw [gray](35) -- (40);
	\draw [gray](37) -- (40);

\node (101)[label={below:A. $0$-link}] [vertex4] at (-9,1) {};
\node (102)[label={below:B. $1$-link}] [vertex4] at (-3,1) {};	
\node (103)[label={below:C. $2$-link}] [vertex4] at (3,1) {};		
\node (104)[label={below:D. $3$-link}] [vertex4] at (9,1) {};
	
	\end{scope}

\end{tikzpicture}

\caption{$3$ horizontal blue-nonblue pairs in $H_{[y,y+1]}$ (nonblue vertices in yellow)}
	\label{fig:blue-nonblue}

\end{figure}


In Lemmas \ref{lem-blue-blue}, \ref{lem-blue-nonblue2} and \ref{lem-blue-nonblue} we identified blue-nonblue pairs within a particular link $H_{[y,y+1]}$. The next two lemmas identify distinguished pairings across link-sets. Let $P \subset [j,j+n-2]$ be a couple set (see definition on page \pageref{dense-sparse}) of size $r$ with corresponding $\alpha_y$-links $H_{[y,y+1]} \le H(n,n)$ for each $y \in P$. 

\begin{lemma}\label{lem:beta-not-dense}
If $\beta$ is not dense in $P$ and each $H_{[y,y+1]}$ for $y \in P$ has an adjacent blue-nonblue pair with the same polarity, then we can combine these pairs to form a matched distinguished pairing $\{U,W\}$ of size $r$ where the vertices of $U$ are blue and the vertices of $W$ nonblue.
\end{lemma}
\begin{proof}
Suppose $s,t \in P$ such that $(v_s,v_{s+1})$ and $(v_t,v_{t+1})$ are two adjacent blue-nonblue pairs in different links, with $v_s, v_t \in U$ and $v_{s+1}, v_{t+1} \in W$. Consider the two possible $\beta$ bonds  $(v_s,v_{t+1})$ and $(v_{s+1},v_t)$. If neither of these bonds exist then $v_s$ is distinguished from $v_t$ by both $v_{s+1}$ and $v_{t+1}$ (see Figure \ref{fig:beta-not-dense}  $(i)$). If one of these bonds exists then $v_s$ is distinguished from $v_t$ by either $v_{s+1}$ or $v_{t+1}$ (see Figure \ref{fig:beta-not-dense} $(ii)$ and $(iii)$). Both bonds cannot exist as $\beta$ is not dense in $P$. Note that the bonds $(v_s,v_t)$ and $(v_{s+1},v_{t+1})$ are not relevant in distinguishing $v_s$ from $v_t$ since, if they exist, they connect blue to blue and nonblue to nonblue.

So any two blue vertices $v_s, v_t \in U$ are distinguished by the two nonblue vertices $v_{s+1}, v_{t+1} \in W$ and hence $\{U,W\}$ is a matched distinguished pairing of size $r$.
\end{proof}

\begin{figure}\centering
\begin{tikzpicture}[scale=0.5,
	vertex2/.style={circle,draw,minimum size=6,fill=blue},
	vertex3/.style={circle,draw,minimum size=6,fill=yellow},
	vertex4/.style={circle,draw=white,minimum size=6,fill=white},]
	
	\node (1)[label={above:$v_s$}] [vertex2] at (-14,3) {};
	\node (2)[label={below:$v_{s+1}$}][vertex3] at (-12,2) {};
	\node (3)[label={above:$v_t$}][vertex2] at (-10,5) {};
	\node (4)[label={below:$v_{t+1}$}][vertex3] at (-8,4) {};
	\node (5)[label={above:$v_s$}] [vertex2] at (-4,3) {};
	\node (6)[label={below:$v_{s+1}$}][vertex3] at (-2,2) {};
	\node (7)[label={above:$v_t$}][vertex2] at (0,5) {};
	\node (8)[label={below:$v_{t+1}$}][vertex3] at (2,4) {};		
	\node (9)[label={above:$v_s$}] [vertex2] at (6,3) {};
	\node (10)[label={below:$v_{s+1}$}][vertex3] at (8,2) {};
	\node (11)[label={above:$v_t$}][vertex2] at (10,5) {};
	\node (12)[label={below:$v_{t+1}$}][vertex3] at (12,4) {};
	
	\draw (1) -- (2); \draw (3) -- (4);
	\draw (7) -- (8) -- (5) -- (6);
	\draw (12) -- (11) -- (10) -- (9);
	
	\node (201)[label={below:$(i)$}] [vertex4] at (-11,1) {};
	\node (202)[label={below:$(ii)$}] [vertex4] at (-1,1) {};	
	\node (203)[label={below:$(iii)$}] [vertex4] at (9,1) {};	
			
	\end{tikzpicture}

\caption{Adjacent blue-nonblue vertex pairs, $\beta$ not dense (nonblue vertices in yellow)}
	\label{fig:beta-not-dense}

\end{figure}
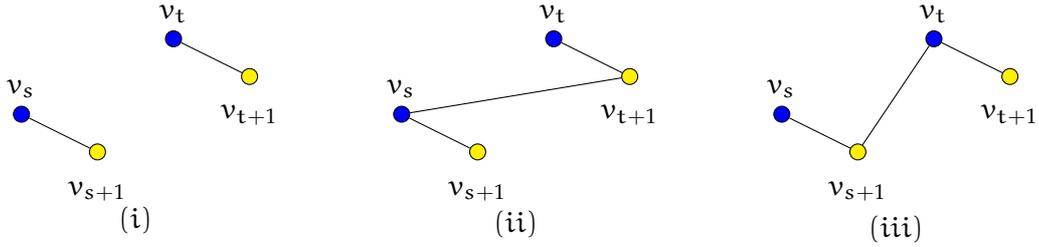


\begin{lemma}\label{lem:beta-not-sparse}
If $\beta$ is not sparse in $P$ and each $H_{[y,y+1]}$ has a non-adjacent blue-nonblue pair with the same polarity, then we can combine these pairs to form an unmatched distinguished pairing $\{U,W\}$ of size $r$ where the vertices of $U$ are blue and the vertices of $W$ nonblue.
\end{lemma}
\begin{proof}
This is very similar to the proof of Lemma \ref{lem:beta-not-dense} and is demonstrated in Figure \ref{fig:beta-not-sparse}.
\end{proof}

\begin{figure}\centering
\begin{tikzpicture}[scale=0.5,
	vertex2/.style={circle,draw,minimum size=6,fill=blue},
	vertex3/.style={circle,draw,minimum size=6,fill=yellow},
	vertex4/.style={circle,draw=white,minimum size=6,fill=white},]	
	
	\node (21)[label={above:$v_s$}] [vertex2] at (-14,3) {};
	\node (22)[label={below:$v_{s+1}$}][vertex3] at (-12,2) {};
	\node (23)[label={above:$v_t$}][vertex2] at (-10,5) {};
	\node (24)[label={below:$v_{t+1}$}][vertex3] at (-8,4) {};
	\node (25)[label={above:$v_s$}] [vertex2] at (-4,3) {};
	\node (26)[label={below:$v_{s+1}$}][vertex3] at (-2,2) {};
	\node (27)[label={above:$v_t$}][vertex2] at (0,5) {};
	\node (28)[label={below:$v_{t+1}$}][vertex3] at (2,4) {};
	\node (29)[label={above:$v_s$}][vertex2] at (6,3) {};
	\node (30)[label={below:$v_{s+1}$}][vertex3] at (8,2) {};
	\node (31)[label={above:$v_t$}][vertex2] at (10,5) {};
	\node (32)[label={below:$v_{t+1}$}][vertex3] at (12,4) {};
			
	\draw (21) -- (24); \draw (22) -- (23); 
	\draw (25) -- (28);
	\draw (30) -- (31);	
	
	\node (206)[label={below:$(i)$}] [vertex4] at (-11,1) {};
	\node (207)[label={below:$(ii)$}] [vertex4] at (-1,1) {};
	\node (208)[label={below:$(iii)$}] [vertex4] at (9,1) {};

\end{tikzpicture}

\caption{Non-adjacent blue-nonblue vertex pairs, $\beta$ not sparse (nonblue vertices in yellow)}
	\label{fig:beta-not-sparse}

\end{figure}
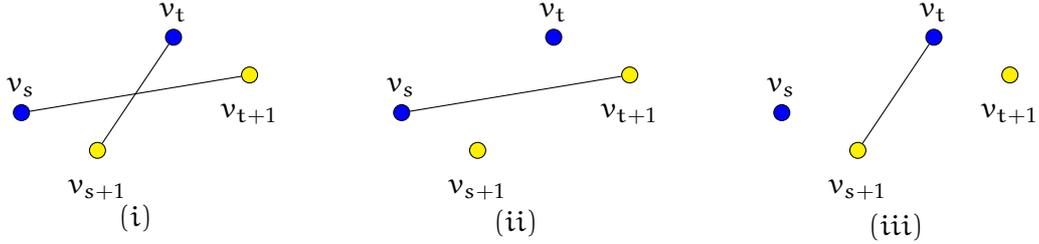
%
%
%
%
\subsection{Two colour partition cases to consider}\label{two_cases}

Having identified structures that give us a lower bound on labels required for a clique-width expression for $H(n,n)$, we now apply this knowledge to the following subtree of $\tree(\tau)$.

Let $\oplus_a$ be the lowest node in $\tree(\tau)$ such that $H_a$  contains all the vertices in rows $2$ to $(n-1)$ in some column of $H(n,n)$. We reserve rows $1$ and $n$ so that we may apply Lemmas \ref{lem-blue-blue} and \ref{lem-blue-nonblue2}.

Thus $H(n,n)$ contains at least one column where vertices in rows $2$ to $(n-1)$ are non-white but no column has entirely blue or red vertices in rows $2$ to $(n-1)$ because otherwise $\oplus_a$  would not be the lowest node in $\tree(\tau)$ such that $H_a$  contains all the vertices in rows $2$ to $(n-1)$ in some column of $H(n,n)$. Let $C_b$ be a non-white column. Without loss of generality we can assume that the number of blue vertices in column $C_b$ between rows $2$ and $(n-1)$ is at least $(n/2)-1$ otherwise we could swap red for blue.

Now consider rows $2$ to $(n-1)$. We have two possible cases:

\begin{description}
\item [Case 1] Either none of the rows with a blue vertex in column $C_b$ has blue vertices in every column to the right of $C_b$, or none of the rows with a blue vertex in column $C_b$ has blue vertices in every column to the left of $C_b$. Hence, we have at least $\lceil n/2 \rceil -1$ rows that have a horizontal blue-nonblue pair with the same polarity.
\item [Case 2] One row $R_r$ has a blue vertex in column $C_b$ and blue vertices in every column to the right of $C_b$ and one row $R_l$ has a blue vertex in column $C_b$ and blue vertices in every column to the left of $C_b$. Hence, either on row $R_r$ or row $R_l$, we must have  a horizontal set of consecutive blue vertices of size at least $\lceil n/2 \rceil +1$. 
\end{description}

To prove unboundedness of clique-width we will show that for any $r \in \mathbb{N}$ we can find an $n \in \mathbb{N}$ so that any clique-width expression $\tau$ for $H(n,n)$ requires at least $r$ labels in $\tree(\tau)$, whether this is a 'Case 1' or 'Case 2' scenario.   

To address both cases we will need the following classic result:

\begin{thm}[Ramsey~\cite{ramsey:ramsey-theory:} and Diestel~\cite{diestel:graph-theory5:}] \label{Ramsey}
For every $r \in \mathbb{N}$, every graph of order at least $2^{2r-3}$ contains either $K^r$ or  $\overline{K^r}$ as an induced subgraph.
\end{thm}

We handle first Case 1, for all values of  $\delta=(\alpha,\beta,\gamma)$. 

\begin{lemma}\label{lem:case2}
For any $\delta=(\alpha,\beta,\gamma)$ and any $r \in \mathbb{N}$, if $n \ge 9 \times 2^{4r-1}$ and $\tau$ is a clique-width expression for $H(n,n)$ that results in Case 1 at node $\oplus_a$, then $\tau$  requires at least $r$ labels to construct $H(n,n)$.
\end{lemma}

\begin{proof}
In Case 1 we have, without loss of generality, at least $\lceil n/2 \rceil -1$ horizontal blue-nonblue vertex pairs but we don't know which links these fall on. 

If there are at least $\sqrt{n/2}$ such pairs on the same link then using Lemma \ref{lem-blue-nonblue} we have a matched distinguished pairing $\{U,W\}$  of size $\sqrt{n/2}-1 > r$ such that $\Colour(u)=$ blue for all $u \in U$ and $\Colour(w)\neq$ blue for all $w \in W$. 

If there is no link with $\sqrt{n/2}$ such pairs then there must be at least one such pair on at least  $\sqrt{n/2}$ different links. From Lemma \ref{lem-blue-nonblue2} each such link contains both an adjacent and non-adjacent blue-nonblue pair. It follows from the pigeonhole principle that there is a subset of these of size $\sqrt{n/2}/4$ where the adjacent blue-nonblue pairs have the same polarity and also the non-adjacent blue-nonblue pairs have the same polarity. We use this subset (Note, the following argument applies whether the blue vertex is on the left or right for the adjacent and non-adjacent pairs). If we take the index of the first column in each link in the mentioned subset, and then take every third one of these, we have a couple set $P$ where $|P| \ge \sqrt{n/2}/12$, with corresponding link set $S_L=\{H_{[y,y+1]}: y \in P\}$, such that the adjacent blue-nonblue pair in each link has the same polarity and the non-adjacent blue-nonblue pair in each link has the same polarity. 

Define the graph $G_P$ so that $V(G_P)=P$ and for $x,y \in V(G_P)$ we have $x \sim y$ if and only if they are $\beta$-dense (see definition on page \pageref{dense-sparse}). From Theorem \ref{Ramsey} for any $r$, as $|P| \ge \sqrt{n/2}/12 \ge 2^{2r-3}$ then there exists a couple set $Q \subseteq P$ such that $G_Q$ is either $K^r$ or $\overline{K^r}$.

If $G_Q$ is $\overline{K^r}$, it follows that $\beta$ is not dense in $Q$, and  $S_L$ contains a link set of size $r$ corresponding to the couple set $Q$ where each link has an adjacent blue-nonblue  pair with the same polarity. Applying Lemma \ref{lem:beta-not-dense} this gives us a matched distinguished pairing $\{U,W\}$ of size $r$ such that $\Colour(u)=$ blue for all $u \in U$ and $\Colour(w)\neq$ blue for all $w \in W$.

If $G_Q$ is $K^r$, it follows that $\beta$ is not sparse in $Q$, and $S_L$ contains a link set of size $r$ corresponding to the couple set $Q$ where each link has a non-adjacent blue-nonblue pair with the same polarity. Applying Lemma \ref{lem:beta-not-sparse} this gives us an unmatched distinguished pairing $\{U,W\}$ of size $r$ such that $\Colour(u)=$ blue for all $u \in U$ and $\Colour(w)\neq$ blue for all $w \in W$.

In each case we can construct a distinguished pairing $\{U,W\}$ of size $r$ such that $\Colour(u)=$ blue for all $u \in U$ and $\Colour(w)\neq$ blue for all $w \in W$. Hence, from Lemma \ref{labels} $\tau$  uses at least $r$ labels to construct $H(n,n)$.
\end{proof}

%
%
%
%
\subsection{When \texorpdfstring{$\alpha$}{} has an infinite number of \texorpdfstring{$2$}{}s or \texorpdfstring{$3$}{}s}

For Case 2 we must consider different values for $\alpha$ separately. We denote $m_{23}(n)$ to be the total number of $2$s and $3$s in $\alpha_{[1,n-1]}$.

\begin{lemma}\label{23-links}
For any triple $\delta=(\alpha,\beta,\gamma)$ and any $r \in \mathbb{N}$, if $m_{23}(n) \ge 3 \times 2^{2r}$ and $\tau$ is a clique-width expression for $H(n,n)$ that results in Case 2 at node $\oplus_a$, then $\tau$  requires at least $r$ labels to construct $H(n,n)$.
\end{lemma}

\begin{proof}
Remembering that $C_b$ is the non-white column, without loss of generality we can assume that there are at least $(m_{23}(n)/2)$ $2$- or $3$-links to the right of $C_b$, since otherwise we could reverse the order of the columns. In Case 2 each link has a horizontal blue-blue vertex pair with at least one nonblue vertex in each column, so using Lemma \ref{lem-blue-blue} we have both an adjacent  and non-adjacent blue-nonblue pair in each of these links. 

It follows from the pigeonhole principle that there is a subset of these of size $(m_{23}(n)/8)$ where the adjacent blue-nonblue pairs have the same polarity and also the non-adjacent blue-nonblue pairs have the same polarity. We use this subset. If we take the index of the first column in each link in the mentioned subset, and then take every third one of these, we have a couple set $P$ where $|P| \ge (m_{23}(n)/24)$, with corresponding link set $S_L=\{H_{[y,y+1]}: y \in P\}$,  such that the adjacent blue-nonblue pair in each link has the same polarity and the non-adjacent blue-nonblue pair in each link has the same polarity.

As in the proof of Lemma \ref{lem:case2}, we define a graph $G_P$ so that $V(G_P)=P$ and for $x,y \in V(G_P)$ we have $x \sim y$ if and only if they are $\beta$-dense. From Theorem \ref{Ramsey} for any $r$, as $|P| \ge (m_{23}(n))/24) \ge 2^{2r-3}$ then there exists a couple set $Q \subseteq P$ such that $G_Q$ is either $K^r$ or $\overline{K^r}$.

We now proceed in an identical way to Lemma \ref{lem:case2} to show that we can always construct a distinguished pairing $\{U,W\}$ of size $r$ such that $\Colour(u)=$ blue for all $u \in U$ and $\Colour(w)\neq$ blue for all $w \in W$. Hence, from Lemma \ref{labels} $\tau$  uses at least $r$ labels to construct $H(n,n)$.
\end{proof}

\begin{cor}\label{23-cor}
For any triple $\delta=(\alpha,\beta,\gamma)$ such that $\alpha$ has an infinite number of $2$s or $3$s the hereditary graph class $\GGG^{\delta}$ has unbounded clique-width.
\end{cor}
\begin{proof}
This follows directly from Lemma \ref{lem:case2} for Case 1 and Lemma \ref{23-links} for Case 2, since for any $r \in \mathbb{N}$ we can choose $n$ big enough so that $n \ge 9 \times 2^{4r-1}$ and $m_{23}(n) \ge 3 \times 2^{2r}$ so that whether we are in Case 1 or Case 2 at node $\oplus_a$ we require at least $r$ labels  for any clique-width expression for $H(n,n)$.
\end{proof}

We are aiming to state our result in terms of unbounded $\mathcal{N}^\delta$ so we also require the following.

\begin{lemma}\label{23-N-unbound}
For any triple $\delta=(\alpha,\beta,\gamma)$ such that $\alpha$ has an infinite number of $2$s or $3$s the parameter $\mathcal{N}^\delta$ is unbounded.
\end{lemma}

\begin{proof}
If there is an infinite number of $2$s in $\alpha$ we can create a couple set $P$ of any required size such that $\alpha_x=2$ for every $x \in P$, so that in the two-row graph (see Section \ref{two-row}) $v_{x,1} \not\sim v_{x+1,2}$ and $v_{x,2} \sim v_{x+1,1}$ (i.e. we have both an adjacent and non-adjacent pair in the $\alpha_x$-link).

We now apply the same approach as in Lemmas \ref{lem:case2} and  \ref{23-links}, applying Ramsey theory to the graph $G_P$ defined in the same way as before. Then for any $r$ we can set $|P| \ge 2^{2r-3}$ so that there exists a couple set $Q \subseteq P$ where $G_Q$ is either  $K^r$ or $\overline{K^r}$.

If $G_Q$ is $\overline{K^r}$ it follows that $\beta$ is not dense in $Q$. So for any $x,y \in Q$,  $v_{x+1,1}$ and $v_{y+1,1}$ have different neighbourhoods in $R_2(Q)$ since they are distinguished by either $v_{x,2}$ or $v_{y,2}$. Hence, if $n$ is the highest natural number in $Q$ then $\mathcal{N}^\delta([1,n+1]) \ge r$.

If $G_Q$ is $K^r$ it follows that $\beta$ is not sparse in $Q$. So for any $x,y \in Q$, $v_{x,1}$ and $v_{y,1}$ have different neighbourhoods in $R_2(Q)$ since they are distinguished by either $v_{x+1,2}$ or $v_{y+1,2}$. Hence, $\mathcal{N}^\delta([1,n+1]) \ge r$.
  
Either way, we have $\mathcal{N}^\delta([1,n+1]) \ge r$, but $r$ can be arbitrarily large, so  $\mathcal{N}^\delta$ is unbounded.

A similar argument applies if there is an infinite number of $3$s.
\end{proof}

%
%
%
%

\subsection{When \texorpdfstring{$\alpha$}{} has a finite number of \texorpdfstring{$2$}{}s and \texorpdfstring{$3$}{}s}

If $\alpha$ contains only a finite number of $2$s and $3$s then there exists $J \in \mathbb{N}$ such that $\alpha_j  \in \{0,1\}$ for $j > J$. In Case 2, where we have a part-row of consecutive blue vertices,  we are interested in the adjacencies of these blue vertices to the nonblue vertices in each column. Although the nonblue vertices could be in any row, in fact,  if $\alpha$ is over the alphabet $\{0,1\}$, the row index of the nonblue vertices does not  alter the blue-nonblue adjacencies. 

In Case $2$, let $Q$ be the set of column indices of the horizontal set of consecutive blue vertices in row $R_r$ of $H(n,n)$ and let $U_1=\{v_{i,r}: i \in Q\}$ be this horizontal set of blue vertices. Let $U_2=\{u_j: j \in Q\}$ be the corresponding set of nonblue vertices such that $u_j \in C_j$. We have the following:

\begin{lemma}\label{lem-two-row}
In Case $2$, with $U_1$ and $U_2$ defined as above, if $\alpha$ is a word over the alphabet $\{0,1\}$  then for any $i,j \in Q$, $v_{i,r} \sim u_j$ in $\PPP^\delta$ if and only if $v_{i,1} \sim v_{j,2}$ in the two-row graph $T^\delta(Q)$.
\end{lemma}
\begin{proof}
Considering the vertex sets $U_1 \cup U_2$ of $\PPP^\delta$ and $R_1(Q) \cup R_2(Q)$ of $T^\delta(Q)$ (see Section \ref{two-row}) we have:

\begin{enumerate}[label=(\alph*)]
	\item For $i=j$ both $v_{j,r} \sim u_j$ and $v_{j,1} \sim v_{j,2}$ if and only if $\gamma_j=1$. 
	\item For $|i-j|>1$ both $v_{i,r} \sim u_j$ and $v_{i,1} \sim v_{j,2}$ if and only if $(i,j) \in \beta$.
	\item For $j=i+1$ both $v_{i,r} \sim u_j$ and $v_{i,1} \sim v_{j,2}$ if and only if $\alpha_i=1$.
\end{enumerate}
Hence $v_{i,r} \sim u_j$ if and only if $v_{i,1} \sim v_{j,2}$.
\end{proof}

\begin{lemma}\label{lem-finite23}
If $\delta=(\alpha,\beta,\gamma)$ where $\alpha$ is an infinite word over the alphabet $\{0,1,2,3\}$ with a finite number of $2$s and $3$s, then the hereditary graph class $\GGG^\delta$ has unbounded clique-width if and only if $\mathcal{N}^\delta$ is unbounded.
\end{lemma}
\begin{proof}
First, we prove that $\GGG^\delta$ has unbounded clique-width if $\mathcal{N}^\delta$ is unbounded.

As $\alpha$ has a finite number of $2$s and $3$s there exists a $J \in \mathbb{N}$ such that ${\alpha}_j \in \{0,1\}$ if $j > J$.

As $\mathcal{N}^\delta$ is unbounded this means that from Lemma \ref{lem:Nbound} for any $r \in \mathbb{N}$ there exist $N_1,N_2 \in \mathbb{N}$ such that, setting $Q_1= [J+1, J+N_1]$ and $Q_2= [J+N_1+1, J+N_1+N_2]$, then $\mathcal{N}^\delta(Q_1) \ge r$ and $\mathcal{N}^\delta(Q_2) \ge r$. 

Denote the $n \times n$ graph $H^\prime(n,n)=H^\delta_{J+1,1}(n,n) \in \GGG^\delta$. As described in Section \ref{two_cases} we again consider the two possible cases for a clique-width expression $\tau$ for  $H^\prime(n,n)$ at a node $\oplus_a$ which is the lowest node in $\tree(\tau)$ such that $H_a$ contains a column of $H^\prime(n,n)$.

Case $1$ is already covered by Lemma \ref{lem:case2} for $n \ge 9 \times 2^{4r-1}$.

In Case $2$, one row $R_r$ of $H^\prime(N_1+N_2,N_1+N_2)$ has a blue vertex in column $C_b$ and blue vertices in every column to the right of $C_b$ and one row $R_l$ has a blue vertex in column $C_b$ and blue vertices in every column to the left of $C_b$. 

If $b \le J+N_1$ then consider the graph to the right of $C_b$. We know every column has a blue vertex in row $R_r$ and a non-blue vertex in a row other than $R_r$. The column indices to the right of $C_b$ includes $Q_2$. It follows from Lemma \ref{lem-two-row}   that in the columns whose indices belong to $Q_2$ the neighbourhoods of the blue set (the mentioned blue vertices) to the non-blue set, are identical to the neighbourhoods in graph $T^\delta(Q_2)$ between the vertex sets $R_1(Q_2)$ and $R_2(Q_2)$. 

On the other hand if $b > J+N_1$ we can make an identical claim for the graph to the left of $C_b$ which now includes the column indices for $Q_1$. It follows from Lemma \ref{lem-two-row} that the neighbourhoods of the blue set to the non-blue set are identical to the neighbourhoods in graph $T^\delta(Q_1)$ between the vertex sets $R_1(Q_1)$ and $R_2(Q_1)$. 

As both $\mathcal{N}^\delta(Q_1)=\mu(T^{\delta}(Q_1),R_1(Q_1)) \ge r$ and $\mathcal{N}^\delta(Q_2)=\mu(T^{\delta}(Q_2),R_2(Q_2)) \ge r$ it follows from Lemma \ref{labels} that any clique-width expression for $H^\prime(n,n)$ with $n \ge (N_1+N_2)$ resulting in Case $2$ requires at least $r$ labels.

For any $r \in \mathbb{N}$ we can choose $n$ big enough so that $n \ge$ max $\{9 \times 2^{4r-1}, (N_1+N_2)\}$ so that whether we are in Case $1$ or Case $2$ at node $\oplus_a$ we require at least $r$ labels  for any clique-width expression for $H^\prime(n,n)$.  Hence, $\GGG^\delta$ has unbounded clique-width if $\mathcal{N}^\delta$ is unbounded. 

Secondly, suppose that  $\mathcal{N}^\delta$ is bounded, so that there exists $N \in \mathbb{N}$ such that $\mathcal{N}^\delta([J+1,n])=\mu(T^\delta([J+1,n]),R_1([J+1,n])) < N$ for all $n > J$ . 

We claim $lcwd(\GGG^\delta) \le 2J+N+2$. For we can create a linear clique-width expression using no more than $2J+N+2$ labels that constructs any graph in $\GGG^\delta$ row by row, from bottom to top and from left to right.

For any graph $G \in \GGG^\delta$ let it have an embedding in the grid $\PPP$ between columns $1$ and $M>J$. 

We will use the following set of $2J+N+2$ labels:
\begin{itemize}
\item 2 \emph{current vertex labels}: $a_1$ and $a_2$; 
\item $J$ \emph{current row labels for first $J$ columns}: $\{c_y : y=1,\dots,J\}$;
\item $J$ \emph{previous row labels for first $J$ columns}: $\{p_{y} : y = 1,\dots,J\}$;
\item $N$ \emph{partition labels}: $\{s_{y} : y=1,\dots,N\}$.
\end{itemize}

We allocate a default partition label $s_y$ to each column of $G_{[J+1, M]}$ according to the $R_2([J+1,M])$-similar equivalence classes of the vertex set $R_1([J+1,M])$ in $T^\delta([J+1,M])$. There are at most $N$ partition sets $\{S_y\}$ of $R_1([J+1,M])$, and if vertex $v_{i,1}$ is in $S_y$, $1 \le y \le N$,  then the default partition label for vertices in column $i$ is $s_y$. It follows that for two default column labels, $s_x$ and $s_y$, vertices in columns with label $s_y$ are  either all adjacent to vertices in columns with label $s_x$ or they are all non-adjacent (except the special case of vertices in consecutive columns and the same row, which will be dealt with separately in our clique-width expression).

Carry out the following row-by-row linear iterative process to construct each row $j$, starting with row $1$.

\begin{enumerate}[label=(\roman*)]
\item Construct the first $J$ vertices in row $j$, label them $c_1$ to $c_J$ and build any edges between them as necessary.
\item Insert required edges from each vertex labelled $c_1,\dots,c_J$ to vertices in lower rows in columns $1$ to $J$. This is possible because the vertices in lower rows in column $i$ ($1 \le i \le J$) all have label $p_i$ and have the same adjacency with the vertices in the current row.
\item Relabel vertices labelled $c_1,\dots,c_J$ to $p_1,\dots,p_{J-1},a_2$ respectively.
\item Construct and label subsequent vertices in row $j$ (columns $J+1$ to $M$), as follows. 
	\begin{enumerate}[label=(\alph*)]
	\item Construct the next vertex in column $i$ and label it 
	$a_1$ (or $a_2$). 
	\item If $\alpha_{i-1}=0$ then insert an edge from the current 
	vertex $v_{i,j}$ (label $a_1$) to the previous vertex 
	$v_{i-1,j}$ (label $a_2$).
	\item Insert edges to vertices that are adjacent as a result 
	of the partition $\{S_y\}$ described above. This is possible 
	because all previously constructed vertices with a particular 
	default partition label $s_y$ are either all adjacent or all 
	non-adjacent to the current vertex.
	\item Insert edges from the current vertex to vertices 
	labelled $p_j$ ($1 \le j \le J$) as necessary.
	\item Relabel vertex $v_{i,j-1}$ to its default partition 
	label $s_y$.
	\item Create the next vertex in row $i$ and label it $a_2$ (or 
	$a_1$ alternating). 
	\end{enumerate}
\item When the end of the row is reached, repeat for the next row.
\end{enumerate}

Hence we can construct any graph in the class with at most $2J+N+2$ labels so the clique-width of $\GGG^\delta$ is bounded if  $\mathcal{N}^\delta$ is bounded.
\end{proof}

Corollary \ref{23-cor}, Lemma \ref{23-N-unbound} and Lemma \ref{lem-finite23} give us the following:

\begin{thm}\label{thm-0123unbound}
For any triple $\delta=(\alpha,\beta,\gamma)$ the hereditary graph class $\GGG^\delta$ has unbounded clique-width if and only if $\mathcal{N}^\delta$ is unbounded.
\end{thm}

We will denote $\Delta$ as the set of all $\delta$-triples for which the class $\GGG^\delta$ has unbounded clique-width.

%
%
%
%
%
%
%

\section{\texorpdfstring{$\GGG^\delta$}{\GGG{}} graph classes that are minimal of unbounded clique-width}\label{Sect:Minimal}

To show that for some $\delta \in \Delta$ the class $\GGG^\delta$  is a minimal class of unbounded clique-width we must show that any proper hereditary subclass $\CCC$ has bounded clique-width. If $\CCC$ is a hereditary graph class such that $\CCC \subsetneq \GGG^\delta$ then there must exist a non-trivial finite forbidden graph $F$ that is in $\GGG^\delta$ but not in $\CCC$. In turn, this graph $F$ must be an induced subgraph of some $H^\delta_{j,1}(k,k)$ for some $j$ and $k\in \mathbb{N}$, and thus $\CCC \subseteq \Free(H^\delta_{j,1}(k,k))$. 

We know that for a minimal class, $\delta$ must be recurrent, because if it contains a $k$-factor $\delta_{[j,j+k-1]}$ that either does not repeat, or repeats only a finite number of times, then $\GGG^\delta$ cannot be minimal, as forbidding the induced subgraph $H^\delta_{j,1}(k,k)$ would leave a proper subclass that still has unbounded clique-width.  Therefore, we will only consider recurrent $\delta$ for the remainder of the paper.

%
%
\subsection{The bond-graph}\label{bond-graph}

To study minimality we will use the following  graph class. A \emph{bond-graph} $B^\beta(Q)=(V,E)$ for finite $Q \subseteq\mathbb{N}$ has vertices $V=Q$ and edges $E=\beta_Q$.

Let $\BBB^\beta= \{B^\beta(Q): Q \subseteq \mathbb{N} \text{ finite}\}$. Note that  $\BBB^\beta$ is a hereditary subclass of $\GGG^\delta$ because
	\begin{enumerate}[label=(\alph*)]
	\item  if $Q^\prime \subseteq Q$ then $B^\beta(Q^\prime)$ is 
	also a bond-graph, and
	\item $B^\beta(Q)$ is an 
	induced subgraph of $\PPP^\delta$ since if $Q=\{y_1, y_2, \dots 
	,y_n\}$ with $y_1<y_2< \dots <y_n$ then it can be constructed 
	from $\PPP^\delta$ by taking one vertex from each column $y_j$ 
	in turn such that there is no $\alpha$ or $\gamma$ edge to 
	previously picked vertices.
	\end{enumerate}

We define a parameter (for $n \ge 2$) 
\[\mathcal{M}^\beta(n)= \sup_{m<n} \mu(B^\beta([1,n]),[1,m]).\]

The bond-graphs can be characterised as the sub-class of graphs on a single row (although missing the $\alpha$-edges) with the parameter $\mathcal{M}^\beta$ measuring the number of distinct neighbourhoods between intervals of a single row.

We will say that the bond-set $\beta$ has \emph{bounded $\mathcal{M}^\beta$} if there exists $M$ such that $\mathcal{M}^\beta(n) < M$ for all $n \in \mathbb{N}$.

The following proposition will prove useful later in creating linear clique-width expressions.
\begin{prop}\label{Mlabels}
Let $n,m,m^\prime \in \mathbb{N}$ satisfy $m < m^\prime < n$. Then for graph $B^\beta([1,n])$, in any partition of $[1,m]$ into $[m+1,n]$-similar sets $\{S_i:1 \le i \le k\}$ and $[1,m^\prime]$ into $[m^\prime+1,n]$-similar sets $\{S^\prime_j:1 \le j \le k^\prime\}$ for every $\ell \in [1,k]$ there exists $\ell^\prime \in [1,k^\prime]$ such that $S_\ell \subseteq S^\prime_{\ell^\prime}$.
\end{prop}
\begin{proof}
As two vertices $x$ and $y$ in $S_\ell$ have the same neighbourhood in $[m+1,n]$ it follows they have the same neighbourhood in $[m^\prime+1,n]$ since $m < m^\prime$ so  $x$ and $y$ must sit in the same $[m^\prime+1,n]$-similar set $S^\prime_{\ell^\prime}$ for some $\ell^\prime \in [1,k^\prime]$.
\end{proof}

\begin{prop}
For any $\delta=(\alpha,\beta,\gamma)$ and any $n \in \mathbb{N}$, \[\mathcal{M}^\beta(n) \le \mathcal{N}^\delta([1,n])+1.\]
\end{prop}
\begin{proof}
In the two-row graph $T^\delta([1,n])$ partition $R_1([1,n])$ into $R_2([1,n])$-similar equivalence classes $\{W_i\}$ so that two vertices $v_{x,1}$ and $v_{y,1}$ are in the same set $W_i$ if they have the same neighbourhood in $R_2([1,n])$. By definition the number of such sets is $\mu(T^\delta([1,n]),R_1([1,n]))=\mathcal{N}^\delta([1,n])$. For $m<n$ partition $[1,m]$ into $s$ sets $\{P_i\}$ such that $P_i=\{j: v_{j,1} \in W_i\}$. Then  $s$ is no more than the number of sets in $\{W_i\}$ by definition, but no less than $\mu(B^\beta([1,n]),[1,m])-1$, the number of equivalence classes that are $[m+1,n]$-similar (excluding, possibly, vertex $m$). This holds for all $m<n$, so
\[\mathcal{M}^\beta(n)-1=\sup_{m<n} \mu(B^\beta([1,n]),[1,m])-1 \le \mu(T^\delta([1,n]),R_1([1,n]))=\mathcal{N}^\delta([1,n]).\]
\end{proof}

%
%
%

\subsection{Veins and Slices}\label{sect:veins_slices}

We will start by considering only graph classes $\GGG^\delta$ for $\delta=(\alpha,\beta,\gamma)$ in which $\alpha$ is an infinite word from the alphabet $\{0,2\}$ and then extend to the case where $\alpha$ is an infinite word from the alphabet $\{0,1,2,3\}$.

Consider a specific embedding of a graph $G=(V,E) \in \CCC$ in $\PPP^\delta$, and recall that the induced subgraph of $G$ on the vertices $V\cap C_{[j,j+k-1]}$ is denoted $G_{[j,j+k-1]}$.
 
Let $\alpha$ be an infinite word over the alphabet $\{0,2\}$. A \emph{vein} $\mathcal{V}$ of $G_{[j,j+k-1]}$ is a set of $t \le k$ vertices $\{v_s,\dots, v_{s+t-1}\}$ in consecutive columns such that $v_y \in V \cap C_y$ for each $y \in \{s, \dots , s+t-1 \}$ and for which $v_y \sim v_{y+1}$ for all $y \in \{s, \dots , s+t-2 \}$. 

We will call a vein of length $k$ a \emph{full} vein and a vein of length $<k$ a \emph{part} vein.  Note that as $\alpha$ comes from the alphabet $\{0,2\}$, for a vein $\{v_s,\dots, v_{s+t-1}\}$, $v_{y+1}$ is no higher than $v_y$ for each $y \in \{s, \dots , s+t-2 \}$. A horizontal row of $k$ vertices in $G_{[j,j+k-1]}$ is a full vein.

As $G$ is $\Free(H^\delta_{j,1}(k,k))$ we know that no set of vertices of $G$ induces $H^\delta_{j,1}(k,k)$.  We consider this in terms of disjoint full veins of $G_{[j,j+k-1]}$. Note that $k$ rows of vertices between column $j$ and column $j+k-1$ are a set of $k$ disjoint full veins and induce a graph isomorphic to $H^\delta_{j,1}(k,k)$. There are other sets of $k$ disjoint full veins that form a graph isomorphic to $H^\delta_{j,1}(k,k)$, but some sets of $k$ full veins do not. Our first task is to clarify when a set of $k$ full veins has this property.

Let $\{v_j,\dots, v_{j+k-1}\}$ be a full vein such that each vertex $v_x$ has coordinates $(x,u_x)$ in $\PPP$, observing that $u_{x+1}\le u_x$ for $x \in [j,j+k-2]$.  We construct an \emph{upper border} to be a set of vertical coordinates $\{w_j,\dots, w_{j+k-1}\}$ using the following procedure:
\begin{enumerate}[label=(\arabic*)]
\item Set $w_j=u_j$,
\item Set $x=j+1$,
	\item if $\alpha_{x-1}=2$ set $w_x=u_{x-1}$, 
	\item if $\alpha_{x-1}=0$ set $w_x=w_{x-1}$,
	\item set $x=x+1$,
	\item if $x=j+k$ terminate the procedure, otherwise return to step $(3)$.
\end{enumerate}

Given a full vein $\mathcal{V}=\{v_j,\dots, v_{j+k-1}\}$, define the \emph{fat vein} $\mathcal{V}^f=\{v_{x,y} \in V(G_{[j,j+k-1]}):x \in [j,j+k-1],y \in [u_x,w_x]\}$ (See examples shown in Figure \ref{fig-col1}).

Let $\mathcal{V}_1$ and $\mathcal{V}_2$ be two full veins. Then we say they are \emph{independent} if $\mathcal{V}^f_1 \cap \mathcal{V}^f_2 = \emptyset$ i.e. their corresponding fat veins are disjoint.

\begin{prop} \label{prop:full_veins}
$G_{[j,j+k-1]}$ cannot contain more than $(k-1)$ independent full veins.
\end{prop}
\begin{proof}
We claim that $k$ independent full veins $\{\mathcal{V}_1, \dots, \mathcal{V}_k\}$ induce the forbidden graph $H^\delta_{j,1}(k,k)$. 

Remembering $v_{x,y}$ is the vertex in the grid $\PPP$ in the $x$-th column and $y$-th row, let $w_{x,y}$ be the vertex in the $y$-th full vein $\mathcal{V}_y$ in column $x$. We claim the mapping $\phi(w_{x,y}) \rightarrow v_{x,y}$ is an isomorphism.

Consider vertices $w_{x,y} \in \mathcal{V}_y$ and $w_{s,t} \in \mathcal{V}_t$ for $t \ge y $. Then
\begin{enumerate}[label=(\alph*)]
\item If $t=y$ (i.e the vertices are on the same vein) then both $w_{x,y} \sim w_{s,t}$ and $v_{x,y} \sim v_{s,t}$ if and only if $|x-s|=1$ or $(x,s) \in \beta$,
\item If $t>y$ and $x=s$ then both $w_{x,y} \sim w_{s,t}$ and $v_{x,y} \sim v_{s,t}$ if and only if $\gamma_x=1$,
\item If $t>y$ and $s=x+1$  then both $w_{x,y} \not\sim w_{s,t}$ and $v_{x,y} \not\sim v_{s,t}$,
\item If $t>y$ and $s=x-1$ then both $w_{x,y} \sim w_{s,t}$ and $v_{x,y} \sim v_{s,t}$ if and only if $\alpha_s=2$,
\item If $t>y$ and $|s-x|>1$ then both $w_{x,y} \sim w_{s,t}$ and $v_{x,y} \sim v_{s,t}$ if and only if $(x,s) \in \beta$.
\end{enumerate}

Hence, $w_{x,y} \sim w_{s,t}$ if and only if $v_{x,y} \sim v_{s,t}$ and $\phi$ is an isomorphism from $k$ independent full veins to $H^\delta_{j,1}(k,k)$.
\end{proof}
%
%
%
%
%
\subsection{Vertex colouring}\label{sect:vert_col}

Our objective is to identify conditions on (recurrent) $\delta \in \Delta$ that make $\GGG^\delta$ a minimal class of unbounded  clique-width. For such a $\delta$ it is sufficient to show that any graph $G$ in a proper hereditary subclass $\CCC$ has bounded  linear clique-width. In order to do this we will partition $G$ into manageable sections (which we will call "panels"), the divisions between the panels chosen so that they can be built separately and then 'stuck' back together again, using a linear clique-width expression requiring only a bounded number of labels. In this section we describe a vertex colouring that will lead (in Section \ref{Sect:panels}) to the construction of these panels. 

As previously observed, for any subclass $\CCC$ there exist $j$ and $k$ such that $\CCC \subseteq \Free(H^\delta_{j,1}(k,k))$.
As $\delta$ is recurrent, if we let $\delta^*=\delta_{[j,j+k-1]}$ be the $k$-factor that defines the forbidden graph  $H^\delta_{j,1}(k,k)$, we can find $\delta^*$ in $\delta$ infinitely often, and we will use these instances of $\delta^*$ to divide our embedded graph $G$ into the required panels. 
 
Firstly, we will construct a maximal set $\mathbb{B}$ of independent full veins for $G_{[j,j+k-1]}$, a section of $G$ that by Proposition \ref{prop:full_veins} cannot have more than $(k-1)$ independent full veins. We start with the lowest full vein (remembering that the rows of the grid $\PPP$ are indexed from the bottom) and then keep adding the next lowest independent full vein until the process is exhausted.  

Note that the next lowest independent full vein is unique because if we have two full veins $\mathcal{V}_1, \mathcal{V}_2$ with vertices 
$\{v_j,\dots, v_{j+k-1}\}$ and $\{v^{\prime}_j,\dots, v^{\prime}_{j+k-1}\}$ respectively then they can be combined to give $\{min(v_j,v^{\prime}_j),\dots,min(v_{j+k-1},v^{\prime}_{j+k-1})\}$ which is a full vein with a vertex in each column at least as low as the vertices of $\mathcal{V}_1$ and $\mathcal{V}_2$.

Let $\mathbb{B}$ contain $b<k$ independent full veins, numbered from the bottom as $\mathcal{V}_1, \cdots, \mathcal{V}_b$ such that any other full vein not in $\mathbb{B}$ must have a vertex in common with a fat vein $\mathcal{V}^f_y$ corresponding to one of the veins $\mathcal{V}_y$ of $\mathbb{B}$.  

Let $u_{x,y}$ be the lowest vertical coordinate and $w_{x,y}$ the highest vertical coordinate of vertices in $\mathcal{V}^f_y \cap C_x$. We define $\mathcal{S}_0 = \{v_{x,y}\in V(G_{[j,j+k-1]}) : x \in [j,j+k-1],  y < u_{x,1} \}$, $\mathcal{S}_b = \{v_{x,y}\in V(G_{[j,j+k-1]}) : x \in [j,j+k-1],  y > w_{x,b} \}$, and for $y=1,\dots,b-1$ we define:

\[ \mathcal{S}_i = \{v_{x,y}\in V(G_{[j,j+k-1]}) : x \in [j,j+k-1], w_{x,i} < y < u_{x,i+1} \}\]

This gives us $b+1$ \emph{slices} $\{\mathcal{S}_0, \mathcal{S}_1, \cdots, \mathcal{S}_b \}$. 

We partition the vertices in the fat veins and the slices into sets which have similar neighbourhoods, which will facilitate the division of $G$ into panels. We colour the vertices of $G_{[j,j+k-1]}$ so that each slice has green/pink vertices to the left and red vertices to the right of the partition, and each fat vein has blue vertices (if any) to the left and yellow vertices to the right. Examples of vertex colourings are shown in Figure \ref{fig-col1}.

Colour the vertices of each slice $\mathcal{S}_i$ as follows:
\begin{itemize}
\item Colour any vertices in the left-hand column green. Now colour green any remaining vertices in the slice that are connected to one of the green left-hand column vertices by a part vein that does not have a vertex in common with any of the fat veins corresponding to the full veins in $\mathbb{B}$. 
\item Locate the column $t$ of the right-most green vertex in the slice. If there are no green vertices set $t=s=j$. If $t>j$ then choose $s$ in the range $j \le s < t$ such that $s$ is the highest column index for which $\alpha_{s}=2$. If there are no columns before $t$ for which $\alpha_{s}=2$ then set $s=j$. Colour pink any vertices in the slice (not already coloured) in columns $j$ to $s$ which are below a vertex already coloured green.
\item Colour any remaining vertices in the slice red. 
\end{itemize}
Note that no vertex in the right-hand column can be green because if there was such a vertex then this would contradict the fact that there can be no full veins other than those which have a vertex in common with one of the fat veins corresponding to the full veins in $\mathbb{B}$.  Furthermore, no vertex in the right hand column can be pink as this would contradict the fact that every pink vertex must lie below a green vertex in the same slice. 

Colour the vertices of each fat vein $\mathcal{V}^f_i$ as follows:
\begin{itemize}
\item Let $s$ be the column as defined above for the slice immediately above the fat vein. If $s=j$ colour the whole fat vein yellow. If $s>j$ colour vertices of the fat vein in columns $j$ to $s$ blue and the rest of the vertices in the fat vein yellow.
\end{itemize}


\begin{figure}
\centering
\begin{tikzpicture}[yscale=0.6,xscale=0.6,
vertex4/.style={circle,draw=white,minimum size=6,fill=white},]

\foreach \x/\y in {1/4,1/15,1/22,2/4,2/21,2/22,3/3,
	3/20,3/21}
		\node[vertex,blue!70!black] (\x-\y) at (\x,\y) {};
	\foreach \x/\y in {2/15,3/12,3/13,3/14,
	4/3,4/11,4/18,4/19,5/3,5/9,5/10,5/11,5/18,
	6/3,6/9,6/18,7/1,7/5,7/2,7/3,7/6,7/7,7/8,7/9,7/16,7/18}
		\node[vertex,yellow!70!black] (\x-\y) at (\x,\y) {};	
	\foreach \x/\y in {1/3,1/5,1/9,1/12,1/17,1/18,1/23,1/24,
	2/1,2/7,2/16,2/23,3/6,3/7,3/22,4/4,4/21}
		\node[vertex,green!70!black] (\x-\y) at (\x,\y) {};
	\foreach \x/\y in {2/19,3/8,3/17,3/24,
	4/1,4/8,4/14,4/24,5/5,5/6,5/7,5/13,5/14,5/22,5/23,5/24,
	6/4,6/12,6/11,6/13,6/14,6/21,7/20,7/21,7/24}
		\node[vertex,red!70!black] (\x-\y) at (\x,\y) {};
		
	{
	\draw[very thick,blue] (1-4)--(2-4)--(3-3)--(4-3)
	--(5-3)--(6,3)--(7-1);
	\draw[very thick,blue] (1-15)--(2-15)--(3-12)--(4-11)
	--(5-9)--(6-9)--(7-5);
	\draw[very thick,blue] (1-22)--(2-21)--(3-20)--(4-18)
	--(5-18)--(6-18)--(7-16);
	\draw[blue] (7-3)--(6-3)--(7-2);
	\draw[blue] (2-15)--(3-14)--(4-11)--(3-13)--(2-15);
	\draw[blue] (4-11)--(5-11)--(6-9)--(5-10)--(4-11);
	\draw[blue] (7-9)--(6-9)--(7-8);.
	\draw[blue] (7-7)--(6-9)--(7-6);
	\draw[blue] (1-22)--(2-22)--(3-21)--(4-18);
	\draw[blue] (2-22)--(3-20)--(4-19);
	\draw[blue] (2-21)--(3-21)--(4-19)--(5-18);
	\draw[blue] (6-18)--(7-18);,
	
	\draw[green] (1-24)--(2-23)--(3-22)--(4-21);
	\draw[green] (1-23)--(2-23);
	\draw[green] (1-18)--(2-16)--(1-17);	
	\draw[green] (1-12)--(2-7)--(1-9);
	\draw[green] (2-7)--(3-7)--(4-4)--(3-6)--(2-7);
	\draw[green] (1-3)--(2-1);
	}

	\node (99)[label={below:Example $1$}] [vertex4] at (4,1) {};

\begin{scope}[shift={(8,0)}]

\foreach \x/\y in {1/4,1/15,1/22,2/4,2/21,2/22,3/4,
	3/20,3/21}
		\node[vertex,blue!70!black] (\x-\y) at (\x,\y) {};
	\foreach \x/\y in {2/15,3/12,3/13,3/14,
	4/3,4/4,4/9,4/10,4/11,4/18,4/19,5/3,5/9,5/10,5/11,5/18,5/19,
	6/3,6/4,6/9,6/11,6/18,6/19,7/3,7/4,7/9,7/11,7/18,7/19}
		\node[vertex,yellow!70!black] (\x-\y) at (\x,\y) {};	
	\foreach \x/\y in {1/3,1/5,1/9,1/12,1/17,1/18,1/23,1/24,
	2/1,2/7,2/16,2/23,3/6,3/7,3/22,4/5,4/21,
	5/5,5/21,6/5,6/21}
		\node[vertex,green!70!black] (\x-\y) at (\x,\y) {};
	\foreach \x/\y in {2/19,3/8,3/17,3/24,
	4/1,4/8,4/14,4/24,
	5/6,5/7,5/13,5/14,5/23,5/24,
	6/12,6/13,6/14,6/24,
	7/1,7/2,7/6,7/7,7/8,7/14,7/16,7/24}
		\node[vertex,red!70!black] (\x-\y) at (\x,\y) {};
		
	{
	\draw[very thick,blue] (1-4)--(2-4)--(3-4)--(4-3)
	--(5-3)--(6,3)--(7-3);
	\draw[very thick,blue] (1-15)--(2-15)--(3-12)--(4-9)
	--(5-9)--(6-9)--(7-9);
	\draw[very thick,blue] (1-22)--(2-21)--(3-20)--(4-18)
	--(5-18)--(6-18)--(7-18);
	\draw[blue] (2-15)--(3-14)--(4-11)--(3-13)--(2-15);
	\draw[blue] (7-9)--(6-9);
	\draw[blue] (1-22)--(2-22)--(3-21)--(4-18);
	\draw[blue] (2-22)--(3-20)--(4-19);
	\draw[blue] (2-21)--(3-21)--(4-19)--(5-19)--
	(6-19)--(7-19);
	\draw[blue] (3-14)--(4-10)--(3-13)--(4-9);
	\draw[blue] (4-11)--(3-12)--(4-10);
	\draw[blue] (4-9)--(3-14);
	\draw[blue] (3-4)--(4-4);
	\draw[blue] (4-10)--(5-10);
	\draw[blue] (4-11)--(5-11)--(6-11)--(7-11);
	\draw[blue] (6-4)--(7-4);
	
	\draw[green] (1-24)--(2-23)--(3-22)--(4-21)
	--(5-21)--(6-21);
	\draw[green] (1-23)--(2-23);
	\draw[green] (1-18)--(2-16)--(1-17);	
	\draw[green] (1-12)--(2-7)--(1-9);
	\draw[green] (2-7)--(3-7)--(4-5)--(3-6)--(2-7);
	\draw[green] (4-5)--(5-5)--(6-5);
	\draw[green] (1-3)--(2-1);	
	}

\node (99)[label={below:Example $2$}] [vertex4] at (4,1) {};

\end{scope}

\begin{scope}[shift={(16,-8)}]
	\foreach \x/\y in {1/24,2/23,3/21,4/12, 
	5/12,6/12,7/12,8/12,9/8,
	10/5}
		\node[vertex,blue!70!black] (\x-\y) at (\x,\y) {};
	\foreach \x/\y in {2/24,3/22,3/23,
	4/13,4/14,4/15,4/16,4/17,4/18,4/19,4/20,4/21,
	5/14,5/15,5/18,5/19,5/20,
	6/13,6/14,6/15,6/18,6/19,7/15,7/16,7/19,7/21,
	8/15,8/19,8/21}
		\node[vertex,blue!70!black] (\x-\y) at (\x,\y) {};
	\foreach \x/\y in {9/8,9/9,9/10,9/11,9/12,
	10/5,10/6,10/7,10/8}
		\node[vertex,yellow!70!black] (\x-\y) at (\x,\y) {};	
	\foreach \x/\y in {1/36,2/35,3/33,4/30, 
	5/30,6/30,7/30,8/30,9/19}
		\node[vertex,green!70!black] (\x-\y) at (\x,\y) {};	
	\foreach \x/\y in {1/35,2/33,3/31,4/28,4/27,4/26,4/25, 
	5/27,6/27,7/27,8/27,9/16}
		\node[vertex,green!70!black] (\x-\y) at (\x,\y) {};	
	\foreach \x/\y in {5/23,5/26,6/29,6/26,6/25,6/24,6/23,7/29,
	7/26,7/25,7/24,7/23,8/24,8/23}
		\node[vertex,pink!70!black] (\x-\y) at (\x,\y) {};
	\foreach \x/\y in {3/36,4/35,5/35,6/35,7/35,8/35,9/34}
		\node[vertex,red!70!black] (\x-\y) at (\x,\y) {};						
	{
	\draw[very thick,blue] (1-24)--(2-23)--(3-21)--(4-12)
	--(5-12)--(6,12)--(7-12)--(8-12)--(9-8)--(10-5);
	\draw[blue] (1-24)--(2-24)--(3-23)--(4-21);
	\draw[blue] (2-23)--(3-23)--(4-20)--(5-20);
	\draw[blue] (2-24)--(3-22)--(4-21);
	\draw[blue] (2-24)--(3-21)--(4-21);
	\draw[blue] (2-23)--(3-22)--(4-20);
	\draw[blue] (3-23)--(4-19)--(5-19)--(6-19)--(7-19)--(8-19);
	\draw[blue] (3-23)--(4-18);\draw[blue] (3-23)--(4-17);
	\draw[blue] (3-23)--(4-16);\draw[blue] (3-23)--(4-15);
	\draw[blue] (3-23)--(4-14);\draw[blue] (3-23)--(4-13);
	\draw[blue] (3-22)--(4-19);
	\draw[blue] (3-22)--(4-18);\draw[blue] (3-22)--(4-17);
	\draw[blue] (3-22)--(4-16);\draw[blue] (3-22)--(4-15);
	\draw[blue] (3-22)--(4-14);\draw[blue] (3-22)--(4-13);
	\draw[blue] (3-21)--(4-20);\draw[blue] (3-21)--(4-19);
	\draw[blue] (3-21)--(4-18);\draw[blue] (3-21)--(4-17);
	\draw[blue] (3-21)--(4-16);\draw[blue] (3-21)--(4-15);
	\draw[blue] (3-21)--(4-14);\draw[blue] (3-21)--(4-13);
	\draw[blue] (4-18)--(5-18)--(6-18);
	\draw[blue] (4-15)--(5-15)--(6-15)--(7-15)--(8-15);
	\draw[blue] (4-14)--(5-14)--(6-14);
	\draw[blue] (8-19)--(9-12);\draw[blue] (8-19)--(9-11);
	\draw[blue] (8-19)--(9-10);\draw[blue] (8-19)--(9-9);
	\draw[blue] (8-19)--(9-8);
	\draw[blue] (8-21)--(9-12);\draw[blue] (8-21)--(9-11);
	\draw[blue] (8-21)--(9-10);\draw[blue] (8-21)--(9-9);
	\draw[blue] (8-21)--(9-8);
	\draw[blue] (8-15)--(9-12);\draw[blue] (8-15)--(9-11);
	\draw[blue] (8-15)--(9-10);\draw[blue] (8-15)--(9-9);
	\draw[blue] (8-15)--(9-8);
	\draw[blue] (8-12)--(9-12);\draw[blue] (8-12)--(9-11);
	\draw[blue] (8-12)--(9-10);\draw[blue] (8-12)--(9-9);
	\draw[blue] (9-12)--(10-8);\draw[blue] (9-12)--(10-7);
	\draw[blue] (9-12)--(10-6);\draw[blue] (9-12)--(10-5);
	\draw[blue] (9-11)--(10-8);\draw[blue] (9-11)--(10-7);
	\draw[blue] (9-11)--(10-6);\draw[blue] (9-11)--(10-5);
	\draw[blue] (9-10)--(10-8);\draw[blue] (9-10)--(10-7);
	\draw[blue] (9-10)--(10-6);\draw[blue] (9-10)--(10-5);
	\draw[blue] (9-9)--(10-8);\draw[blue] (9-9)--(10-7);
	\draw[blue] (9-9)--(10-6);\draw[blue] (9-9)--(10-5);
	\draw[blue] (9-8)--(10-8);\draw[blue] (9-8)--(10-7);
	\draw[blue] (9-8)--(10-6);

	\draw[green] (1-36)--(2-35)--(3-33)--(4-30)
	--(5-30)--(6,30)--(7-30)--(8-30);
	\draw[green] (1-36)--(2-33)--(1-35)--(2-35);
	\draw[green] (2-35)--(3-31)--(2-33)--(3-33);
	\draw[green] (3-33)--(4-28)--(3-31)--(4-30);
	\draw[green] (3-33)--(4-27)--(3-31)--(4-26)--(3-33)--
	(4-25)--(3-31);
	\draw[green] (4-27)--(5-27)--(6-27)--(7-27)--(8-27);
	\draw[green] (9-16)--(8-30)--(9-19);
	\draw[green] (9-16)--(8-27)--(9-19);
	\draw[pink] (5-26)--(6-26)--(7-26);
	\draw[pink] (6-25)--(7-25);
	\draw[pink] (6-29)--(7-29);
	\draw[pink] (6-24)--(7-24)--(8-24);
	\draw[pink] (5-23)--(6-23)--(7-23)--(8-23);
		}

\node (99)[label={below:Example $3$}] [vertex4] at (5,9) {};
	
\end{scope}	
	
\end{tikzpicture}\par
\caption{Examples of vein and slice colouring -- a $222222$, a $222000$ and a $222000022$ factor, with vertices coloured blue, green, pink, red and yellow as described. The only edges shown are the veins (bold blue), other edges in the fat veins (blue), part veins that start on the left column but do not reach the right column (green) and related pink rows.}
\label{fig-col1}
\end{figure}
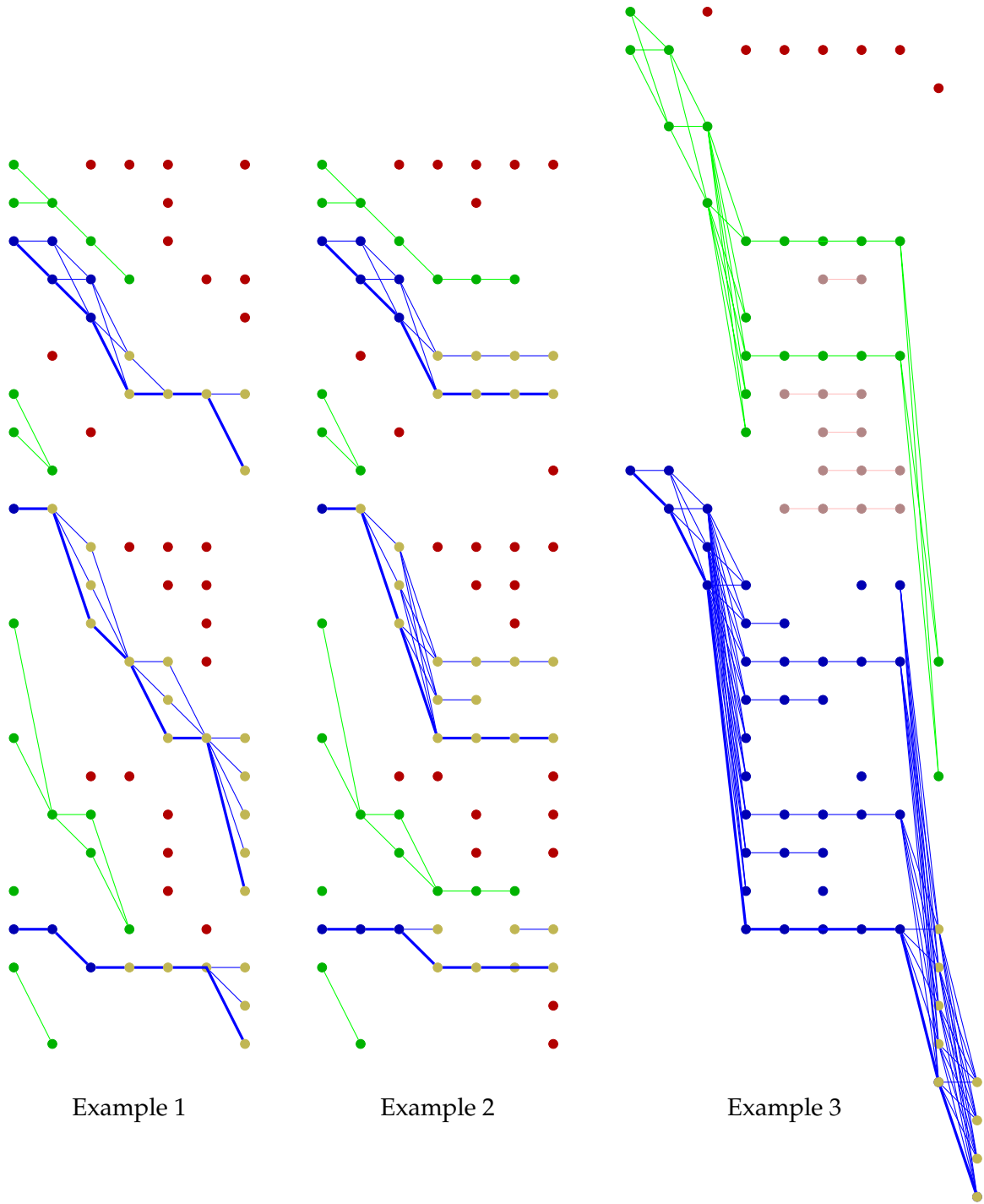


When we create a clique-width expression we will be particularly interested in the edges between the blue and green/pink vertices to the left and the red and yellow vertices to the right.

\begin{prop}\label{red}
Let $v$ be a red vertex in column $x$ and slice $\mathcal{S}_i$.

If $u$ is a blue, green or pink vertex in column $x-1$ then 
\[uv \in E(G) \text{ if and only if } \alpha_{x-1}=2 \text{ and } u \in  \mathcal{V}^f_{i+1} \cup \mathcal{S}_{i+1} \cup \cdots \cup \mathcal{V}^f_{b} \cup \mathcal{S}_b. \]

Similarly, if $u$ is a blue, green or pink vertex in column $x+1$ then 
\[uv \in E(G) \text{ if and only if } \alpha_{x}=2 \text{ and } u \in \mathcal{S}_{0} \cup \mathcal{V}^f_{1} \cup \mathcal{S}_{1} \cup \cdots \cup \mathcal{V}^f_{i} \cup \mathcal{S}_{i}. \]
\end{prop}
\begin{proof}
Note that as $u$ and $v$ are in consecutive columns we need only consider $\alpha$-edges. 

If $u$ is green in column $x-1$ of $\mathcal{S}_i$ then red $v$ in column $x$ of $\mathcal{S}_i$ cannot be adjacent to $u$ as this would place red $v$ on a green part-vein which is a contradiction. Likewise, if $u$ is green in column $x+1$ of $\mathcal{S}_i$ then red $v$ in column $x$ of $\mathcal{S}_i$ must be adjacent to $u$ since if it was not adjacent to such a green vertex in the same slice then this implies the existence of a green vertex above the red vertex in the same column which contradicts the colouring rule to colour pink any vertex in columns $j$ to $s$ below a vertex coloured green.

The other adjacencies are straightforward.
\end{proof}

\begin{prop}\label{yellow}
Let $v$ be a yellow vertex in column $x$ and fat vein $\mathcal{V}^f_i$.

If $u$ is a blue, green or pink vertex in column $x-1$ then 
\[uv \in E(G) \text{ if and only if } \alpha_{x-1}=2 \text{ and } u \in  \mathcal{V}^f_{i} \cup \mathcal{S}_{i} \cup \cdots \cup \mathcal{V}^f_{b} \cup \mathcal{S}_b. \]

Similarly, if $u$ is a blue, green or pink vertex in column $x+1$ then 
\[uv \in E(G) \text{ if and only if } \alpha_{x}=2 \text{ and } u \in \mathcal{S}_{0} \cup \mathcal{V}^f_{1} \cup \mathcal{S}_{1} \cup \cdots \cup \mathcal{V}^f_{i-1} \cup \mathcal{S}_{i-1}. \]
\end{prop}
\begin{proof}
Note that as $u$ and $v$ are in consecutive columns we need only consider $\alpha$-edges. 

If $u$ is blue in column $x-1$ of $\mathcal{V}^f_i$ then yellow $v$ in column $x$ of $\mathcal{V}^f_i$ must be adjacent to $u$ from the definition of a fat vein. Equally, from the colouring definition for a fat vein  there cannot be a blue vertex in column $x+1$ of $\mathcal{V}^f_i$ if there is a yellow vertex in column $x$ of $\mathcal{V}^f_i$.

The other adjacencies are straightforward.
\end{proof}

Having established these propositions, as the pink and green vertices in a particular slice and column have the same adjacencies to the red and yellow vertices, we now combine the green and pink sets and simply refer to them all as \emph{green}.

%
%
%
%
%
\subsection{Extending \texorpdfstring{$\alpha$}{alpha} to the \texorpdfstring{$4$}{4}-letter alphabet}

Our analysis so far has been based on $\alpha$ being a word from the alphabet $\{0,2\}$. We now use the following lemma to extend our colouring to the case where $\alpha$ is a word over the $4$-letter alphabet $\{0,1,2,3\}$.

Let $\alpha$ be an infinite word over the alphabet $\{0,1,2,3\}$ and $\alpha^+$ be the infinite word over the alphabet $\{0,2\}$ such that for each $x \in \mathbb{N}$, 

\[
\alpha^+_x=
\begin{cases}
0&\text{ if } \alpha_x=0 \text{ or } 1,\\
2&\text{ if } \alpha_x=2 \text{ or } 3,
\end{cases}
\]

Denoting $\delta=(\alpha,\beta,\gamma)$ and $\delta^+=(\alpha^+,\beta,\gamma)$, let $G=(V,E)$ be a graph in the class $\GGG^\delta$ with a particular embedding in the vertex grid $V(\PPP)$. We will refer to $G^+=(V,E^+)$ as the graph with the same vertex set $V$ as $G$ from the class $\GGG^{\delta^+}$.

\begin{lemma}\label{lem-U-sim0123}
For any subset of vertices $U \subseteq V$, 2 vertices of $U$ in the same column of $V(\PPP)$ are $V\setminus U$-similar in $G$ if and only if they are $V\setminus U$-similar in $G^+$.
\end{lemma}
\begin{proof}
Let $u_1$ and $u_2$ be two vertices in $U$ in the same column $x$ and $v$ be a vertex of $V\setminus U$ in column $y$. If $x=y$ then $v$ is in the same column as $u_1$ and $u_2$  and is either adjacent to both or neither depending on whether there is a $\gamma$-clique on column $x$, which is the same in both $G$ and $G^+$. If $|x-y|>1$ then $v$ is adjacent to both $u_1$ and $u_2$ if and only if there is a bond $(x,y)$ in $\beta$, which is the same in both $G$ and $G^+$. 

If $y=x+1$ then the adjacency of $v$ to $u_1$ and $u_2$ is determined by $\alpha_x$ in $G$ and $\alpha^+_x$ in $G^+$. If $\alpha_{x}=\alpha^+_{x}$ (i.e. both $0$ or both $2$) then the adjacencies are the same in $G$ and $G^+$. If $\alpha_{x}=1$ and $\alpha^+_{x}=0$, then $u_1$ and $u_2$ are both adjacent to $v$ in $G$ if and only if they are both non-adjacent to $v$ in $G^+$. If $\alpha_{x}=3$ and $\alpha^+_{x}=2$, then $u_1$ and $u_2$ are both adjacent to $v$ in $G$ if and only if they are both non-adjacent to $v$ in $G^+$. 

Hence $u_1$ and $u_2$ have the same neighbourhood in $V\setminus U$ in $G$ if and only if they have the same neighbourhood in $V\setminus U$ in $G^+$.
\end{proof}

\begin{lemma}\label{blue_green}
For a graph $G \in \mathcal{G}^\delta \cap \Free(H^\delta_{j,1}(k,k))$ and $G^+$ defined as above, let the vertices of $G^+_{[j,j+k-1]}$ be coloured as per Section \ref{sect:vert_col}. Then the same colouring applied to the vertices of $G_{[j,j+k-1]}$ has the property that a column of $G_{[j,j+k-1]}$ can be partitioned into at most $k-1$ disjoint blue sets and $k$ disjoint green sets, so that any red or yellow vertex is either adjacent to all or none of a given green/blue vertex set. 
\end{lemma}
\begin{proof}
As $\alpha^+$ is a word over the alphabet $\{0,2\}$ the results of  Sections \ref{sect:veins_slices} and \ref{sect:vert_col} can be applied, in particular Propositions \ref{prop:full_veins}, \ref{red} and \ref{yellow}. It  follows that for $G^+_{[j,j+k-1]}$:
\begin{itemize}
\item  there are no more than $(k-1)$ independent full veins, and consequently at most $k$ slices,
\item two blue vertices in the same fat vein and column have the same red/yellow neighbourhood, and
\item two green vertices in the same slice and column have the same red/yellow neighbourhood. 
\end{itemize}
Lemma \ref{lem-U-sim0123}, with $U^b$ and $U^g$ being the blue and green vertices respectively, and $U=U^b \cup U^g$, tells us that these statements also apply to $G_{[j,j+k-1]}$ and the result follows.
\end{proof}

\subsection{Panel construction}\label{Sect:panels}

We now construct the panels of $G$ based on our embedding of $G$ in $\PPP^{\delta}$.

To recap, $\delta^*=\delta_{[j,j+k-1]}$ is the $k$-factor that defines the forbidden graph  $H^\delta_{j,1}(k,k)$ and we will use the repeated instances of $\delta^*$ to divide our embedded graph $G$ into panels.

Define $t_0, t_1, \dots, t_{z}$ where $t_0$ is the index of the column before the first column of the embedding of $G$, $t_{z}$ is the index of the last column of the embedding of $G$ and $t_i$ ($0<i<z$) represents the rightmost letter index of the $i$-th copy of $\delta^*$ in $\delta$ , such that $t_i>k+t_{i-1}$ to ensure the copies are disjoint. Hence, the $i$-th disjoint copy of $\delta^*$  in $\delta$ corresponds to columns $C_{[t_i-k+1,t_i]}$ of $\PPP^\delta$ and we denote the induced graph on these columns $G_i=G_{[t_i-k+1,t_i]}$ and denote $G^+_i$ as the corresponding graph in $G^+$. 

Colour the vertices of $G^+_i$ blue, yellow, green or red as described in Section \ref{sect:vert_col} and then apply the same colouring to the vertices of $G_i$. Call these $G_i$ vertex sets  $U^b_i$, $U^y_i$, $U^g_i$ and $U^r_i$ respectively. Denote $U^w_1$ as the vertices in $G_{[t_0+1,t_1-k]}$, and for $1 < i < z$ denote $U^w_i$ the set of vertices in $G_{[t_i+1,t_{i+1}-k]}$ and colour the vertices in each $U^w_i$ white.

We now create a sequence of \emph{panels}, the first panel is $P_1=U^w_1 \cup U^g_1 \cup U^b_1$, and subsequent panels given by

\[P_i=U^y_{i-1}\cup U^r_{i-1}\cup U^w_i\cup U^g_i\cup U^b_i.\]

These panels create a disjoint partition of the vertices of our embedding of $G$. The following lemma will be used to put a bound on the number of labels required in a linear clique-width expression to create edges between panels. We denote $\mathbb{P}_i=\cup_{s=1}^i P_s$.

\begin{lemma}\label{lem-Usim-panels}
Let  $(\alpha,\beta,\gamma)$ be a recurrent $\delta$-triple where $\alpha$ is an infinite word over the alphabet $\{0,1,2,3\}$, $\gamma$ is an infinite binary word and $\beta$ is a bond set which has bounded $\mathcal{M}^\beta$, so that $\mathcal{M}^\beta(n) < M$ for all $n \in \mathbb{N}$. 

Then for any graph $G=(V,E) \in \GGG^\delta \cap Free(H^\delta_{j,1}(k,k))$ for some $j,k \in \mathbb{N}$ with vertices $V$ partitioned into panels $\{P_1,\dots,P_z\}$ and $1 \le i \le z$, 
\[\mu(G,V \setminus \mathbb{P}_i) < M + 2k^2.\]
\end{lemma}
\begin{proof}
Considering the three sets of vertices $\mathbb{P}_i \setminus (U^b_{i} \cup U^g_{i})$, $U^b_{i}$ and $U^g_{i}$ in graph $G$ separately, we have:
\begin{enumerate}[label=(\alph*)]
\item the number of distinct neighbourhoods of the vertex set $V \setminus \mathbb{P}_i$ in the vertex set $\mathbb{P}_i \setminus (U^b_{i} \cup U^g_{i})$ is bounded by $M$.
\item the number of distinct neighbourhoods of the vertex set $V \setminus \mathbb{P}_i$ in the vertex set $U^b_{i}$ is bounded by $k(k-1)$, noticing that from Lemma \ref{blue_green} two blue vertices in the same fat vein and column have the same neighbourhood in $V \setminus \mathbb{P}_i$.
\item the number of distinct neighbourhoods of the vertex set $V \setminus \mathbb{P}_i$ in the vertex set $U^g_{i}$ is bounded by $k(k-1)$, noticing that from Lemma \ref{blue_green} two green vertices in the same slice and column have the same neighbourhood in $V \setminus \mathbb{P}_i$.
\end{enumerate}
This covers all vertices of $\mathbb{P}_i$ so
\begin{align*}
\mu(G,V \setminus \mathbb{P}_i) \le M + k(k-1)+ k(k-1)<  M + 2k^2. 
\end{align*}
\end{proof}

%
%
%
%

\subsection{When \texorpdfstring{$\GGG^\delta$}{\GGG{}} is a minimal class of unbounded clique-width}\label{minimalproof} 

Our strategy for proving that an arbitrary graph $G$ in a proper hereditary subclass of $\GGG^\delta$ has bounded linear clique-width (and hence bounded clique-width) is to define an algorithm to create a linear clique-width expression that allows us to recycle labels so that we can put a bound on the total number of labels required, however many vertices there are in $G$. We do this by constructing a linear clique-width expression for each panel $P_i$ in $G$ in a linear sequence, leaving the labels on each vertex of previously constructed panels $\mathbb{P}_{i-1}$ with an appropriate label to allow edges to be constructed between the current panel/vertex and previous panels. To  be able to achieve this we require the following ingredients:
\begin{enumerate}[label=(\alph*)]
\item $\delta$ to be recurrent so we can create the panels,
\item a bound on the number of labels required to create each new panel,
\item a process of relabelling so that we can leave appropriate labels on  each vertex of the current panel to enable connecting to previous panels, before moving on to the next panel, and
\item a bound on the number of labels required to create edges to previously constructed panels.  
\end{enumerate}

We have $(a)$ by assumption and we will deal with $(c)$ and $(d)$ in the proof of Theorem \ref{thm-0123minim}. The next two lemmas show how we can restrict $\delta$ further, using a new concept of 'gap factors', to ensure $(b)$ is achieved.

\begin{lemma}\label{lem-lcw-rectangular-block}
For any $\delta$ and graph $G\in \GGG^\delta$ and any $j_1,j_2 \in \mathbb{N}$ where $|j_2-j_1|=\ell-1$
\[lcw(G_{[j_1,j_2]})\le 2\ell\]
\end{lemma}
\begin{proof}
We construct $G_{[j_1,j_2]}$ using a row-by-row linear method, starting in the bottom left. For each of the $\ell$ columns, we create 2 labels: one label $c_1,\dots,c_{\ell}$ for the vertex in the \emph{current} row being constructed, and one label $e_1,\dots,e_{\ell}$ for the vertices in all \emph{earlier} rows. 

For the first row, we insert the (max) $\ell$ vertices using the labels $c_1,\dots,c_{\ell}$, and since every vertex has its own label we can insert all necessary edges. Now relabel $c_i \to e_i$ for each $i$.

Suppose that the first $r$ rows have been constructed, in such a way that every existing vertex in column $i$ has label $e_i$. We insert the (max) $\ell$ vertices in row $r+1$ using labels $c_1,\dots,c_{\ell}$. As before, every vertex in this row has its own label, so we can insert all edges between vertices within this row. Next, note that any vertex in this row has the same relationship with all vertices in rows $1,\dots,r$ of any column $i$. Since these vertices all have label $e_i$ and the vertex in row $r+1$ has its own label, we can add edges as determined by $\alpha$, $\beta$ and $\gamma$ as necessary. Finally, relabel $c_i \to e_i$ for each $i$, move to the next row and repeat until all rows have been constructed.
\end{proof}


We will call a factor of a $\delta$-triple between, and including, some consecutive disjoint pair of occurrences of a $k$-factor $\delta^*=\delta_{[j,j+k-1]}$, a \emph{$\delta^*$-gap factor}. An \emph{$\mathcal{N}^{\delta}$-bounded recurrent} $\delta$-triple is a recurrent triple where, for any factor $\delta^*$ and any  $\delta^*$-gap factor $\delta_Q$, the value of $\mathcal{N}^{\delta}(Q)$ is bounded by a function of $\delta^*$ only (i.e. it is bounded irrespective of the $\delta^*$-gap factor chosen). In particular, from Lemma \ref{23-N-unbound}, it follows that if $\delta$ is $\mathcal{N}^{\delta}$-bounded recurrent then there is a bound on the number of $2$s and $3$s in the $\alpha$ component  of any $\delta^*$-gap factor.

If $\delta$ is almost periodic, so that for any factor $\delta^*$ of $\delta$ every factor of $\delta$ of length at least  $\mathcal{L}(\delta^*)$ contains $\delta^*$, then each $\delta^*$-gap factor $\delta_Q$ covers a maximum of $\mathcal{L}(\delta^*)+k$ columns. As a consequence of Lemma \ref{lem-lcw-rectangular-block}, $\mathcal{N}^{\delta}(Q)$ is bounded by $2(\mathcal{L}(\delta^*)+k)$ (i.e a function of $\delta^*$ only) irrespective of the $\delta^*$-gap factor chosen.
Hence, every almost periodic $\delta$-triple is also $\mathcal{N}^{\delta}$-bounded recurrent. 

In addition, we know there exist $\mathcal{N}^{\delta}$-bounded recurrent $\delta$-triples which are not almost periodic. In \cite{brignall_cocks:uncountable:} a recurrent but not almost periodic binary word $\psi$ was constructed by a process of substitution. If we take $\delta=(\psi, \emptyset, 0^\infty)$,  then we have an example of an $\mathcal{N}^{\delta}$-bounded recurrent $\delta$-triple that is not almost periodic .

\begin{lemma}\label{lem:delta-gap}
Let $\delta$ be an $\mathcal{N}^{\delta}$-bounded recurrent triple with $k$-factor $\delta^*=\delta_{[j,j+k-1]}$. Then for any graph $G \in \GGG^\delta$, where $V[G] \subseteq C_{Q}$ where $Q$ is an interval such that $\delta_Q$ is a factor of a $\delta^*$-gap factor, there exists a bound on the linear clique-width of $G$ that is a function of $\delta^*$ only.
\end{lemma}
\begin{proof} 
As $\delta$ is an $\mathcal{N}^{\delta}$-bounded recurrent triple there exists a bound $N(\delta^*)$ on $\mathcal{N}^{\delta}(Q)$, where $Q$ is any interval such that $\delta_Q$ is a subset of a $\delta^*$-gap factor.
It follows from Lemma \ref{23-N-unbound} that there is a bound, say $J(\delta^*)$, on the number of $2$s and $3$s in the $\alpha$ factor of any $\delta^*$-gap factor $\delta_{Q}$. 

We can use the row-by-row linear method from the proof of Lemma \ref{lem-finite23} to show that for any graph $G\in \GGG^\delta$, with $V[G] \subseteq C_{Q}$ we have $lcw(G)\le 2J+N+2$.
\end{proof}

We are now in a position to define a set of hereditary graph classes $\GGG^\delta$ that are minimal of unbounded clique-width. We will denote $\Delta_{min} \subseteq \Delta$ as the set of all $\delta$-triples in $\Delta$ with the characteristics:
\begin{enumerate}[label=(\alph*)] 
\item $\delta$ is $\mathcal{N}^{\delta}$-bounded recurrent, and
\item the bond set $\beta$ has bounded $\mathcal{M}^\beta$.  
\end{enumerate}

\begin{thm}\label{thm-0123minim}
If $\delta \in \Delta_{min}$ then $\GGG^\delta$ is a minimal hereditary class of both unbounded linear clique-width and unbounded clique-width. 
\end{thm}

\begin{proof}
$\GGG^\delta$ has unbounded clique-width since $\delta \in \Delta$. We now show that if $\delta \in \Delta_{min}$ then every proper hereditary subclass $\CCC \subsetneq \GGG^\delta$ has bounded linear clique-width. From the introduction to this section we know that for such a subclass $\CCC$ there must exist some $H^\delta_{j,1}(k,k)$ for some $j$ and $k\in \mathbb{N}$ such that $\CCC \subseteq \Free(H^\delta_{j,1}(k,k))$. 

Using the same column indices $\{t_i\}$ used for panel construction of a graph $G \in \GGG^\delta$ in Section \ref{Sect:panels}, let the $i$-th $\delta^*$-gap factor be denoted $\delta_{q_i}$ where $q_1=[t_{0}+1,t_1]$ and $q_i=[t_{i-1}-k+1,t_i]$ for $1<i<z$. Note that for every $i$, $P_i\subseteq C_{q_i}$. 
From Lemma \ref{lem:delta-gap} we know there exist $J$ and $N\in \mathbb{N}$, each a function of $\delta^*$ only, such that the number of labels required to construct each panel $P_i$ by the row-by-row linear method for all $i \in \mathbb{N}$ is no more than $2J+N+2$.

As the bond-set $\beta$ has bounded $\mathcal{M}^\beta$, let $M\in \mathbb{N}$ be a constant such that $\mathcal{M}^\beta(n) < M$ for all $n \in \mathbb{N}$. 

Although a single panel $P_i$ can be constructed using at most $2J+N+2$ labels, we need to be able to recycle labels so that we can construct any number of panels with a bounded number of labels. We will show that any graph $G \in \Free(H^\delta_{j,1}(k,k))$ can be constructed by a linear clique-width expression that only requires a number of labels determined by the constants $M$, $N$, $J$ and $k$. 

For our construction of panel $P_i$, we will use the following set of $4k^2+MN+M+2J+2$ labels:

\begin{itemize}
\item $2$ \emph{current vertex labels}: $a_1$ and $a_2$;
\item $J$ \emph{current row labels}: $\{c_y : y=1,\dots,J\}$  for first $J$ columns;
\item $J$ \emph{previous row labels}: $\{p_{y} : y = 1,\dots,J\}$ for first $J$ columns;
\item $MN$ \emph{partition labels}: $\{s_{x,y} : x=1,\dots,M,y=1,\dots,N\}$, for vertices in $U^y_{i-1} \cup U^r_{i-1} \cup U^w_i$;
\item $k^2$ \emph{blue current panel labels}: $\{bc_{x,y} : x=1,\dots,k, y = 1,\dots,k\}$, for vertices $\mathcal{V}^f_{i,x} \cap U^b_{i} \cap C_y$;
\item $k^2$ \emph{blue previous panel labels}: $\{bp_{x,y} : x=1,\dots,k, y = 1,\dots,k\}$, for vertices $\mathcal{V}^f_{i-1,x} \cap U^b_{i-1} \cap C_y$;
\item $k^2$ \emph{green current panel labels}: $\{gc_{x,y} : x=0,\dots,k-1,y = 1,\dots,k\}$, for vertices $\mathcal{S}_{i,x} \cap U^g_{i} \cap C_y$;
\item $k^2$ \emph{green previous panel labels}: $\{gp_{x,y} : x=0,\dots,k-1, y = 1,\dots,k\}$, for vertices $\mathcal{S}_{i-1,x} \cap U^g_{i-1} \cap C_y$;
\item $M$ \emph{bond labels}: $\{m_{y} : y=1,\dots, M\}$, for vertices in previous panels for creating the $\beta$-bond edges between columns.
\end{itemize}

We carry out the following iterative process to construct each panel $P_i$ in turn.

Assume $\mathbb{P}_{i-1}=\cup_{s=1}^{i-1} P_s$ has already been constructed such that labels $m_y$, $bp_{x,y}$ and $gp_{x,y}$ have been assigned to the $M+2k^2$ $V \setminus \mathbb{P}_{i-1}$-similar sets as described in Lemma \ref{lem-Usim-panels}.

Using the same column indices $\{t_i\}$ used for panel construction (Section \ref{Sect:panels}) we assign a default partition label $s_{x,y}$ to each column of $U^y_{i-1} \cup U^r_{i-1} \cup U^w_i$ as follows: 
\begin{enumerate}[label=(\alph*)]
\item Consider the bond-graph $B^\beta([1,t_z])$  (Section \ref{bond-graph}). We partition the interval $Q=[t_{i-1}-k+1,t_i-k]$ into $[t_i-k+1,t_z]$-similar sets of which there are at most $M$, and use label index $x$ to identify values in $Q$ in the same $[t_i-k+1,t_z]$-similar set. Consequently, vertices in two columns of $U^y_{i-1} \cup U^r_{i-1} \cup U^w_i$ that have the same default label $x$ value have the same neighbourhood in $G_{[t_i-k+1,t_z]}$ and hence are in the same $V \setminus \mathbb{P}_{i}$-similar set.
\item Consider the two-row graph $T^\delta(Q)$ (Section \ref{two-row}). We partition vertices in $R_1(Q)$ into $R_2(Q)$-similar sets of which there are at most $N$. We create a corresponding partition of the interval $Q$ such that $v_{x,1}$ and $v_{y,1}$ are in the same equivalence class of $R_1(Q)$ if and only if $x$ and $y$ are in the same partition set of $Q$. We now use label index $y$ to identify values in the same partition set. Consequently, vertices in two columns of $U^y_{i-1} \cup U^r_{i-1} \cup U^w_i$ that have the same default label $y$ value have the same neighbourhood within $G_{Q}$.
\end{enumerate}

We construct each panel $P_i$ in the row-by-row linear method used for the graph with a finite number of $2$s and $3$s with bounded  $\mathcal{N}^\delta$ constructed in Lemma \ref{lem-finite23}. The current vertex always has a unique label. Thus, for each row, we use labels $c_1, \dots, c_J$ for vertices in the first $J$ columns and then alternate $a_1$ and $a_2$ for the current and previous vertices for the remainder of the row. 

For each new vertex in the current row we add edges as follows:
\begin{enumerate}[label=(\alph*)]
\item Insert required edges to the $\mathcal{M}^\beta+2k^2$ $V \setminus \mathbb{P}_{i-1}$-similar sets -- see Lemma \ref{lem-Usim-panels}. This is possible because vertices within each of these sets are either all adjacent to the current vertex or none of them are.

\item Insert required edges to vertices in the same or lower rows in the current panel. This is possible as these vertices all have labels $p_y$, $s_{x,y}$, $bc_{x,y}$ or $gc_{x,y}$ and, from the construction, vertices with the same $y$ value are either all adjacent to the current vertex or none of them are.

\end{enumerate}

Following completion of edges to the current vertex, we relabel the previous vertex as follows:
\begin{itemize}
	\item from $c_y$ to $p_y$ if it is in the first $J$ columns,
	\item	from $a_2$ (or $a_1$) to its default partition label $s_{x,y}$ if it is in $U^y_{i-1} \cup U^r_{i-1} \cup U^w_i$ but not in the first $J$ columns.
	\item from $a_2$ (or $a_1$) to $bc_{x,y}$ if it is in $\mathcal{V}^f_{i,x} \cap U^b_{i}$, and 
	\item from $a_2$ (or $a_1$) to $gc_{x,y}$ if it is in $\mathcal{S}_{i,x} \cap U^g_{i}$.
\end{itemize}

We now repeat for the next row of panel $P_i$.

Once panel $P_i$ is complete, relabel as follows:
 
Relabel vertices in accordance with their $V \setminus \mathbb{P}_{i}$-similar set, of which there are at most $M$. Note from Proposition \ref{Mlabels}, that two vertices with the same label $m_y$ from the previous $\mathbb{P}_{i-1}$ partition sets will still need the same label in  $\mathbb{P}_{i}$. Two equivalence classes from the  $\mathbb{P}_{i-1}$ partition may merge to form a new equivalence class in the $\mathbb{P}_{i}$ partition.  Hence, it is possible to relabel with the same label the old equivalence classes that merge, and then use the spare $m_y$ labels for any new equivalence classes that appear. We never need more than $M$ such labels. 

Also relabel all vertices with labels $bp_{x,y}$, $gp_{x,y}$, $p_y$ and  $s_{x,y}$  with the relevant bond label $m_y$ of their $V \setminus \mathbb{P}_{i}$-similar set. This is possible for the vertices labelled $s_{x,y}$ as the index $x$ signifies their $V \setminus \mathbb{P}_{i}$-similar set.

Now relabel $bc_{x,y}\rightarrow bp_{x,y}$  and $gc_{x,y}\rightarrow gp_{x,y}$ ready for the next panel. For the next panel we can reuse labels $a_1$, $a_2$, $c_y$, $p_y$, $s_{x,y}$, $bc_{x,y}$ and $gc_{x,y}$ as necessary.

This process repeated for all panels completes the construction of $G$. 

The maximum number of labels required to construct any graph $G \in \Free(H^\delta_{j,1}(k,k))$ is $4k^2+MN+M+2J+2$ and hence $\CCC$ has bounded linear clique-width.
\end{proof}

The conditions for $\delta$ to be in $\Delta_{min}$ are sufficient for the class $\GGG^\delta$ to be minimal. It is fairly easy to see that it is necessary for $\delta$ to be bounded recurrent. However, there remains a question regarding the necessity of the bond set $\beta$ to have bounded $\mathcal{M}^\beta$. We have been unable to identify any $\delta \not\in \Delta_{min}$ such that $\GGG^\delta$ is a minimal class of unbounded clique-width, hence:

\begin{conj}
The hereditary graph class $\GGG^\delta$ is minimal of unbounded clique-width if and only if $\delta \in \Delta_{min}$.
\end{conj}

%
%
%
%
%
%

\section{Examples of new minimal classes} \label{Sect:examples}

It has already been shown in \cite{brignall_cocks:uncountable:} that there are uncountably  many minimal hereditary classes of graphs of unbounded clique-width. However, armed with the new framework we can now identify many other types of minimal classes. Some examples of $\delta=(\alpha,\beta,\gamma$) values that yield a minimal class are shown in Table \ref{table-examples}. 

\begin{table}[ht!]
\renewcommand{\arraystretch}{1.5}
\centering
\begin{tabular}{|m{1.75cm}|m{2.0cm}|m{6.5cm}|m{1.0cm}|m{2.0cm}|} 
\hline
\textbf{Example}  &\textbf{$\alpha$} &\textbf{$\beta$} \text{      } \textbf{($x,y \in \mathbb{N}$)} &\textbf{$\gamma$} &$\mathcal{M}^\beta$ bound\\
\hline
1. &$0^\infty$ &$\emptyset$ &$1^\infty$ & 1 \\
\hline
2. &$1^\infty$ &$(1,x+2)$ &$0^\infty$ & 2\\
\hline
3. &$(23)^\infty$ &$(x,x+2)$ &$0^\infty$ & 3\\
\hline
4. &$0^\infty$ &$(x,y): |x-y| \neq 1, x-y \equiv 1 \pmod{2}$ &$0^\infty$ & 3\\
\hline
5. &$1^\infty$ &$(x,y): x \neq y, x-y \equiv 0 \pmod{2}$ &$1^\infty$ & 2\\
\hline
6. &$2^\infty$ &$(x,y): 1<|x-y|\le n$ (fixed $n$)  &$0^\infty$ & n\\
\hline
\end{tabular}
\caption{New minimal hereditary graph classes of unbounded clique-width}
\label{table-examples}
\end{table}

%
%
%
%
\section{Concluding remarks}\label{Sect-Conclude}

The ideas of periodicity and recurrence are well established concepts when applied to symbolic sequences (i.e. words). Application to $\delta$-triples and in particular $\beta$-bonds is rather different and needs further investigation.

The $\beta$-bonds have been defined as generally as possible, allowing a bond between any two non-consecutive columns. The purpose of this was to capture as many minimal classes in the framework as possible. However, it may be observed that the definition is so general that for any finite graph $G$ it is possible to define $\beta$ so that $G$ is isomorphic to an induced subgraph of  $B^\beta(Q)$ and hence $\GGG^\delta$.

In these $\GGG^\delta$ graph classes we have seen that unboundedness of clique-width is determined by the unboundedness of a parameter measuring the number of distinct neighbourhoods between two-rows. The minimal classes are those which satisfy defined recurrence characteristics and for which there is a bound on a parameter measuring the number of distinct neighbourhoods between vertices in one row.    

Hence, whilst we have created a framework for many types of minimal classes, there may be further classes 'hidden' in the $\beta$-bonds. Indeed, we believe other types of minimal hereditary classes of unbounded clique-width exist and this is still an open area for research.

%
%
%
%
%
%

\paragraph{Acknowledgements} We are grateful to the anonymous referees whose careful reading of an earlier draft led to several significant improvements.

%
%
%
%

\bibliographystyle{plain}
\bibliography{refs}

\end{document}